\documentclass[10pt]{amsart}
\numberwithin{equation}{section}
\usepackage{amsmath}
\usepackage{amscd,amsthm,amssymb,amsfonts}
\usepackage{mathrsfs}
\usepackage{dsfont}
\usepackage{stmaryrd}
\usepackage{euscript}
\usepackage{expdlist}
\usepackage{enumerate}

\input xy
\newtheorem{thm}{Theorem}[section]

\newtheorem{hypA}{Hypothesis}

\newtheorem*{thm*}{Theorem}

\newtheorem{lm}[thm]{Lemma}

\newtheorem*{cor*}{Corollary}
\newtheorem{prop}[thm]{Proposition}
\newtheorem*{conj*}{Conjecture}
\newtheorem{conj}{Conjecture}

\theoremstyle{remark}

\newtheorem*{remark}{Remark}
\theoremstyle{definition}
\newtheorem*{defn*}{Definition}

\newtheorem{I_Remark*}{Remark}

\newcommand{\nc}{\newcommand}

\newcommand{\beq}{\begin{equation}}
\newcommand{\eeq}{\end{equation}}
\newcommand{\bpmx}{\begin{pmatrix}}
\newcommand{\epmx}{\end{pmatrix}}
\newcommand{\bbmx}{\begin{bmatrix}}
\newcommand{\ebmx}{\end{bmatrix}}
\newcommand{\wh}{\widehat}

\def\parref#1{\ref{#1}}
\def\thmref#1{Theorem~\parref{#1}}

\def\propref#1{Proposition~\parref{#1}}
     
\def\secref#1{\S\parref{#1}}

\def\lmref#1{Lemma~\parref{#1}}

\def\makeop#1{\expandafter\def\csname#1\endcsname
  {\mathop{\rm #1}\nolimits}\ignorespaces}
\makeop{Hom}   \makeop{End}   \makeop{Aut}   %\makeop{Isom}
\makeop{Pic} \makeop{Gal}       \makeop{Div} \makeop{Lie}
\makeop{PGL}   \makeop{Corr} \makeop{PSL} \makeop{sgn} \makeop{Spf}
 \makeop{Tr} \makeop{Nr} \makeop{Fr} \makeop{disc}
\makeop{Proj} \makeop{supp} \makeop{ker}   \makeop{Im} \makeop{dom}
\makeop{coker} \makeop{Stab} \makeop{SO} \makeop{SL} \makeop{SL}
\makeop{Cl}    \makeop{cond} \makeop{Br} \makeop{inv} \makeop{rank}
\makeop{id}    \makeop{Fil} \makeop{Frac}  \makeop{GL} \makeop{SU}
\makeop{Trd}   \makeop{Sp} \makeop{Tr}    \makeop{Trd} \makeop{Res}
\makeop{ind} \makeop{depth} \makeop{Tr} \makeop{st} \makeop{Ad}
\makeop{Int} \makeop{tr}    \makeop{Sym} \makeop{can} \makeop{SO}
\makeop{torsion} \makeop{GSp} \makeop{Tor}\makeop{Ker} \makeop{rec}
\makeop{Ind} \makeop{Coker}
 \makeop{vol} \makeop{Ext} \makeop{gr} \makeop{ad}
 \makeop{Gr}\makeop{corank} \makeop{Ann}
\makeop{Hol} %Holomorphic
\makeop{Fitt} \makeop{Mp} \makeop{CAP}

\def\makebb#1{\expandafter\def
  \csname bb#1\endcsname{{\mathbb{#1}}}\ignorespaces}
\def\makebf#1{\expandafter\def\csname bf#1\endcsname{{\bf
      #1}}\ignorespaces}
\def\makegr#1{\expandafter\def
  \csname gr#1\endcsname{{\mathfrak{#1}}}\ignorespaces}
\def\makescr#1{\expandafter\def
  \csname scr#1\endcsname{{\EuScript{#1}}}\ignorespaces}
\def\makecal#1{\expandafter\def\csname cal#1\endcsname{{\mathcal
      #1}}\ignorespaces}
\def\doLetters#1{#1A #1B #1C #1D #1E #1F #1G #1H #1I #1J #1K #1L #1M
                 #1N #1O #1P #1Q #1R #1S #1T #1U #1V #1W #1X #1Y #1Z}
\def\doletters#1{#1a #1b #1c #1d #1e #1f #1g #1h #1i #1j #1k #1l #1m
                 #1n #1o #1p #1q #1r #1s #1t #1u #1v #1w #1x #1y #1z}
\doLetters\makebb   \doLetters\makecal  \doLetters\makebf
\doLetters\makescr
\doletters\makebf   \doLetters\makegr   \doletters\makegr
\def\pfname{\medbreak\noindent{\scshape Proof.}\enspace}%
\renewcommand{\proofname}{\pfname}

    \def\setminus{\smallsetminus}

\normalsize

\makeop{Ram} \makeop{Rep} \makeop{mass}

\makeop{Bl}
\def\abs#1{\left|#1\right|}
\def\norm#1{\lVert#1\rVert}
\def\cB{\EuScript B}

\def\cH{{\mathcal H}}

\def\cS{{\mathcal S}}

\def\cV{{\mathcal V}}

\def\sS{\mathscr S}

\def\bbI{\mathbb I}

\def\bbL{\mathbb L}

\newcommand{\A}{\mathbf A}    % for adele

\newcommand{\per}{\bot}
\def\ot{\otimes}

\def\hookto{\hookrightarrow}
\def\longto{\longrightarrow}
\def\ol{\overline}  \nc{\opp}{\mathrm{opp}} \nc{\ul}{\underline}

\def\XYmatrix{\xymatrix@M=8pt} % make \xymatrix not too cluttered
\def\ncmd{\newcommand}
\ncmd{\xysubset}[1][r]{\ar@<-2.5pt>@{^(-}[#1]\ar@<2.5pt>@{_(-}[#1]}
\ncmd{\XYmatrixc}[1]{\vcenter{\XYmatrix{#1}}}
\ncmd{\xyto}[1][r]{\ar@{->}[#1]}
\ncmd{\xyinj}[1][r]{\ar@{^(->}[#1]}
\ncmd{\xysurj}[1][r]{\ar@{->>}[#1]}
\ncmd{\xyline}[1][r]{\ar@{-}[#1]}
\ncmd{\xydotsto}[1][r]{\ar@{.>}[#1]}
\ncmd{\xydots}[1][r]{\ar@{.}[#1]}
\ncmd{\xyleadsto}[1][r]{\ar@{~>}[#1]}
\ncmd{\xyeq}[1][r]{\ar@{=}[#1]} \ncmd{\xyequal}[1][r]{\ar@{=}[#1]}
\ncmd{\xyequals}[1][r]{\ar@{=}[#1]}
\ncmd{\xymapsto}[1][r]{l\ar@{|->}[#1]}\ncmd{\xyimplies}[1][r]{\ar@{=>}[#1]}
\ncmd{\xyiso}{\ar[r]_-{\sim}}
\def\injxy{\ar@{^(->}}

\newcommand{\pMX}[4]{\begin{pmatrix}
{#1}& {#2}\\
{#3}&{#4}\end{pmatrix} }

\newcommand{\seesaw}[4]{{#1}\ar@{-}[rd]\ar@{-}[d]&{#2}\ar@{-}[d]\\
{#3}\ar@{-}[ru]&{#4}}

\def\x{{\times}}

\def\e{\varepsilon} % episilon factor

\newcommand\stt[1]{\left\{#1\right\}}

\def\l{{\ell}}

\renewcommand\pmod[1]{\,(\mbox{mod }{#1})}
\renewcommand\mod[1]{\,\mbox{mod }{#1}}

\renewcommand\Re{\text{Re}\,}
\def\split{{\mathrm{split}}}

\usepackage{color}
\usepackage{hyperref}
\usepackage{MnSymbol}
\usepackage{mathdots}
\usepackage{tikz-cd}
\usepackage{adjustbox}
\newcommand{\BigWedge}{\mathord{\adjustbox{valign=B,totalheight=.6\baselineskip}{$\bigwedge$}}}

\def\SO{{\rm SO}}
\def\GL{{\rm GL}}
\def\B{{\rm B}}

\def\Mat{{\rm Mat}}
\def\b{\bar}
\def\t{\tilde}
\def\h{\hat}
\def\a{\alpha}
\def\la{\lambda}
\def\Wedge{\BigWedge}

\makeatletter
\@namedef{subjclassname@2020}{\textup{2020} Mathematics Subject Classification}
\makeatother

\title{Rankin-Selberg integrals for $\SO_{2n+1}\x\GL_r$ attached to Newforms and Oldforms}
\author{Yao Cheng}
\address{No. 151, Yingzhuan Road, Tamsui District, New Taipei City 251, Taiwan (R.O.C),  Lui-Hsien Memorial 
Science Hall.}
\email{briancheng@mail.tku.edu.tw}
\begin{document}
\maketitle

\begin{abstract}
The conjectural newform theory for generic representations of $p$-adic $\SO_{2n+1}$ was formulated by P.-Y. Tsai in 
her thesis in which Tsai also verified the conjecture when the representations are supercuspidal. 
The main purpose of this work is to compute the Rankin-Selberg integrals for $\SO_{2n+1}\x\GL_r$ with $1\le r\le n$ 
attached to newforms and also oldforms under the validity of the conjecture.
\end{abstract}

\section{Introduction}
The theory of newforms was first developed by Atkin-Lehner (\cite{AtkinLehner1970}) in the context of modular forms, 
and then by Casselman (\cite{Casselman1973}) in the context of generic representations of $p$-adic $\GL_2$. 
In the modular form setting, newforms are cusp forms of level $\Gamma_0(N)$ which are simultaneously 
eigenfunctions of all Hecke operators. Consequently, their Fourier coefficients satisfy strong recurrence relations and 
hence the associated $L$-functions admit an Euler product expansion. 
One also has the notion of oldforms. These are cusp forms obtained from newforms of lower level via certain level 
raising procedures.\\

In the local $p$-adic setting, we are given an irreducible smooth generic (complex) representation $(\pi,\cV_\pi)$ 
of $\GL_2$ over a $p$-adic field $F$, and a descending family of open compact 
subgroups $\Gamma_0(\frak{p}^m)$ of $\GL_2(F)$ indexed by integers $m\ge 0$ with 
$\Gamma_0(\frak{o})=\GL_2(\frak{o})$. Here $\frak{o}$ is the valuation ring of $F$ and $\frak{p}$ is the maximal ideal
of $\frak{o}$. Then Casselman proved that there exists $a_\pi\ge 0$ such that $\cV_\pi^{\Gamma_0(\frak{p}^m)}=0$ if
 $0\le m<a_\pi$ and ${\rm dim}_\bbC\cV_\pi^{\Gamma_0(\frak{p}^{a_\pi})}=1$. Nonzero elements in 
$\cV_\pi^{\Gamma_0(\frak{p}^{a_\pi})}$ are called newforms of $\pi$. Oldforms are elements in 
$\cV_\pi^{\Gamma_0(\frak{p}^m)}$ with $m>a_\pi$, and can be obtained from newforms via certain 
level raising operators. Casselman's results were subsequently extended to generic representations of $\GL_r$ by 
Jacquet, Piatetski-Shapiro and Shalika (\cite{JPSS1981}, \cite{Jacquet2012}, see also \cite{Matringe2013}) and 
Reeder (\cite{Reeder1991}). Since $\Gamma_0(\frak{o})=\GL_2(\frak{o})$ (which also holds for $r>2$), newforms can 
be viewed as a generalization of spherical elements in unramified representations to arbitrary generic representations. 
In addition to this, they also possess the following important features which lead to their applications in 
number theory and representation theory. First, the integer $a_\pi$ is encoded in the exponent of the $\epsilon$-factor of 
$\pi$. Second, if one computes the zeta integrals in \cite{JLbook}, \cite{JPSS1983} attached to newforms, then one 
obtains the $L$-factor of $\pi$ (\cite{Matringe2013}, \cite{Miyauchi2014}). In particular, the Whittaker functional 
of $\pi$ is nontrivial on the space of newforms. And third, explicit formulae for Whittaker functions associated to newforms 
on the diagonal matrices in $\GL_r(F)$ can be obtained (\cite{Shintani1976}, \cite{Schmidt2002}, \cite{Miyauchi2014}).\\

In \cite{RobertsSchmidt2007}, Roberts-Schmidt generalized Casselman's results to generic representations of 
${\rm GSp}_4$ in the same spirit. Their results are based on what they called the $paramodular$ $subgroups$ 
$K(\frak{p}^m)$. These are open compact subgroups of ${\rm GSp}_4(F)$ indexed by integers $m\ge 0$ with 
$K(\frak{o})={\rm GSp}_4(\frak{o})$. Note that; however, paramodular subgroups do not form a descending chain. 
Now let $(\pi,\cV_\pi)$ be an irreducible smooth $generic$ representation of ${\rm GSp}_4(F)$ with $trivial$ 
central character. Roberts-Schmidt proved that there exists an integer $a_\pi\ge 0$ such that the spaces 
$\cV_\pi^{K(\frak{p}^{m})}=0$ if $0\le m<a_\pi$ and ${\rm dim}_\bbC\cV_\pi^{K(\frak{p}^{a_\pi})}=1$. 
Again, nonzero elements in $\cV_\pi^{K(\frak{p}^{a_\pi})}$ are called newforms of $\pi$, and 
elements in $\cV_\pi^{K(\frak{p}^m)}$ with $m>a_\pi$ are called oldforms. They also showed that the spaces
$\cV_{\pi}^{K(\frak{p}^m)}$ (with $m>a_\pi$) admit the following bases
\begin{equation}\label{E:oldform}
\theta'^i\theta^j\eta^k (v_\pi)\quad\text{for integers $i,j,k\geq 0$ with $i+j+2k=m-a_\pi$}.
\end{equation} 	
Here $v_\pi$ is a newform of $\pi$ and $\theta, \theta':\cV_\pi^{K(\frak{p}^m)}\to\cV_\pi^{K(\frak{p}^{m+1})}$ and 
$\eta:\cV_\pi^{K(\frak{p}^m)}\to\cV_\pi^{K(\frak{p}^{m+2})}$ are the level raising operators defined in their book,
see also \S\ref{SSS:level raising}. In particular, the dimension of $\cV_\pi^{K(\frak{p}^m)}$ can be easily computed.\\

Roberts-Schmidt in addition computed Novodvorsky's zeta integrals for ${\rm GSp}_4\x\GL_1$ (\cite{Novodvorsky1979}) 
attached to newforms and also oldforms. They showed that 
\begin{equation}\label{E:zeta newform n=2}
Z(s,v_\pi)=\Lambda_{\pi,\psi}(v_\pi)L(s,\pi)\neq 0
\end{equation}
and 
\begin{equation}\label{E:zeta oldform n=2}
Z(s,\theta(v))=q^{-s+\frac{3}{2}}Z(s,v),
\quad
Z(s,\theta'(v))=qZ(s,v),
\quad
Z(s,\eta(v))=0
\end{equation}
for $v\in\cV_\pi^{K(\frak{p}^m)}$ (any $m\ge 0$). Here $q$ is the cardinality of the residue field of $F$,
$\Lambda_{\pi,\psi}$ is the Whittaker functional on $\pi$ (depending on an additive character $\psi$ of $F$), 
$Z(s,v)$ is the Novodvorsky's zeta integral attached to $v$, and $L(s,\pi)$ is the $L$-factor of $\pi$ defined by 
Novodovorsky's local integrals. In particular, \eqref{E:zeta newform n=2} implies that the Whittaker functional is 
nontrivial on the space of newforms, and one can use \eqref{E:oldform}, \eqref{E:zeta newform n=2} and 
\eqref{E:zeta oldform n=2} to compute $Z(s,v)$ for every $paramodular$ $vector$ $v$ of $\pi$, i.e. $v$ is 
contained in $\cV_\pi^{K(\frak{p}^m)}$ for some $m$. The integer $a_\pi$ also has meaning. It appears in the 
exponent of the $\epsilon$-factor $\epsilon(s,\pi,\psi)$ of $\pi$ defined by Novodvorsky's zeta integrals.\\

By the accidental isomorphisms $\SO_3\cong{\rm PGL_2}$ and $\SO_5\cong{\rm PGSp}_4$, results of
Casselman and Roberts-Schmidt can be viewed as the theory of newforms for $\SO_3$ and $\SO_5$ respectively. 
In her Harvard thesis (\cite{Tsai2013}), Tsai defined a family of open compact subgroups $K_{n,m}$ 
(cf. \S\ref{SS:paramodular subgroup}) of (split) $\SO_{2n+1}(F)$ such that $K_{1,,m}\cong\Gamma_0(\frak{p}^m)$ 
and $K_{2,m}\cong K(\frak{p}^m)$, and extended the theory of newforms to irreducible smooth generic $supercuspidal$ 
representations of $\SO_{2n+1}(F)$ (for all $n$). Based on this, Tsai then proposed a conjectural newform theory for 
generic representations of $\SO_{2n+1}(F)$ in \cite{Tsai2013}, \cite{Tsai2016}. We will review the conjecture in more 
details in the next subsection (modulo the definition of $K_{n,m}$). The local conjecture has a global counterpart, 
which was formulated by Gross (\cite{Gross2015}) with an eye toward giving a refinement of the global Langlands 
correspondence for discrete symplectic motives of rank $2n$ over $\bbQ$.\\ 

The aim of this work is to investigate the newform conjecture and compute the (local) Rankin-Selberg integrals 
for $\SO_{2n+1}\x\GL_r$ with $1\le r\le n$ attached to (conjectural) newforms and oldforms in generic 
representations of $\SO_{2n+1}(F)$ under the Hypothesis \ref{H} stated in \S\ref{SS:main}. 
These Rankin-Selberg integrals were developed by Gelbart and Piatetski-Shapiro (\cite{GPSR1987}), Ginzburg 
(\cite{Ginzburg1990}) and Soudry (\cite{Soudry1993}), and had already played an important role in Tsai's work. 
Now let us describe our results in more details.

\subsection{Newform conjecture}
Let $G_n=\SO_{2n+1}$ be the split odd special orthogonal of rank 
$n\ge 1$ defined over $F$ and $U_n\subset G_n$ be the upper triangular maximal unipotent subgroup.
Let $\psi$ be an additive character of $F$ with $\ker(\psi)=\frak{o}$. Then we have a non-degenerate character 
$\psi_{U_n}: U_n(F)\to\bbC^\x$ (cf. \S\ref{SSS:SO_{2n+1}}). Let $(\pi,\cV_\pi)$ be an irreducible smooth generic 
representation of $G_n(F)$. We fix a nonzero Whittaker functional $\Lambda_{\pi,\psi}\in 
{\rm Hom}_{U_n(F)}(\pi,\psi_{U_n})$. By the results of Jiang and Soudry (\cite{JiangSoudry2003}, 
\cite{JiangSoudry2004}), we can attach to $\pi$ an $L$-parameter $\phi_\pi$. Then after composing $\phi_\pi$ with
the inclusion $^LG_n^0(\bbC)={\rm Sp}_{2n}(\bbC)\hookto\GL_{2n}(\bbC)$, we get the $\epsilon$-factor 
$\epsilon(s,\phi_\pi,\psi)$ (\cite{Tate1979}) associated to $\phi_\pi$ and $\psi$, which can be written as
\begin{equation}\label{E:epsilon for pi}
\epsilon(s,\phi_\pi,\psi)=\varepsilon_\pi q^{-a_\pi(s-\frac{1}{2})}
\end{equation} 
for some $\varepsilon_\pi\in\stt{\pm 1}$ and an integer $a_\pi\ge 0$.
Since $\pi$ is generic, one has $\cV_\pi^{K_{n,m}}\neq 0$ for some $m\ge 0$ (cf. \lmref{L:existence}). 
Then we have the following conjecture due to Tsai (\cite[Conjecture 1.2.8]{Tsai2013}, \cite[Conjecture 7.8]{Tsai2016}).

\begin{conj}[Tsai]\label{C1}
\noindent
\begin{itemize}
\item[(1)]
The dimension of $\cV_\pi^{K_{n,m}}$ is given by 
\begin{equation}\label{E:dim formula}
{\rm dim}_\bbC\cV_\pi^{K_{n,m}}
=
\begin{pmatrix}
n+\lfloor\frac{m-a_\pi}{2}\rfloor\\n
\end{pmatrix}
+
\begin{pmatrix}
n+\lfloor\frac{m-a_\pi+1}{2}\rfloor-1\\n
\end{pmatrix}.
\end{equation}
In particular, $\cV_\pi^{K_{n,m}}=0$ if $0\le m<a_\pi$ and $\cV_\pi^{K_{n,a_\pi}}$ is one-dimensional.
\item[(2)]
The action of the quotient group $J_{n,a_\pi}/K_{n,a_\pi}$ on $\cV_\pi^{K_{n,a_\pi}}$ gives $\varepsilon_\pi$.
\item[(3)]
The Whittaker functional $\Lambda_{\pi,\psi}$ is nontrivial on $\cV^{K_{n,a_\pi}}_\pi$.
\end{itemize}
\end{conj}

Here $J_{n,m}\supset K_{n,m}$ is another family of open compact subgroups of $G_n(F)$ with
$J_{n,0}=K_{n,0}=G_n(\frak{o})$ and such that $[J_{n,m}:K_{n,m}]=2$ when $m>0$ 
(cf. \S\ref{SS:paramodular subgroup}).

\begin{remark}
Conjecture \ref{C1} holds when $n=1,2$ by the results of Casselman (\cite{Casselman1973}) and Roberts-Schmidt 
(\cite{RobertsSchmidt2007}), see also \lmref{L:conj for n=2}. On the other hand, suppose that $\pi$ is supercuspidal 
(any $n\ge 1$). Then in \cite{Tsai2013}, Tsai proved that \eqref{E:dim formula} holds for $0\le m\le a_\pi$. 
For $m>a_\pi$, she showed that ${\rm dim}_\bbC\cV_\pi^{K_{n,m}}$ is greater than or equal to the RHS of 
\eqref{E:dim formula}, and she conjectured that they should be equal. In addition, Conjecture \ref{C1} $(2)$ and $(3)$ 
for $\pi$ were also verified by Tsai.
\end{remark}

\begin{remark}
As far as we know, Tsai did not formulate the conjecture in this form (in the literature). In fact, she only gave partial 
conjectures in \cite[Conjecture 1.2.8]{Tsai2013} and \cite[Conjecture 7.8]{Tsai2016}. However, she did point out that all 
these are expected to be true for all $\pi$ in the same references. 
\end{remark}

\subsection{Rankin-Selberg integrals}
Let $Z_r\subset\GL_r$ be the upper triangular maximal unipotent subgroup and $\b{\psi}_{Z_r}:Z_r(F)\to\bbC^\x$ be a 
non-degenerate character (cf. \eqref{E:psi_Z}). Let $(\tau,\cV_\tau)$ be a smooth representation of $\GL_r(F)$ with 
finite length. We assume that the space ${\rm Hom}_{Z_r(F)}(\tau,\b{\psi}_{Z_r})$ is one-dimensional in which we fix a 
nonzero element $\Lambda_{\tau,\b{\psi}}$. Let $H_r=\SO_{2r}$ be the split even special orthogonal group defined over 
$F$ and $Q_r\subset H_r$ be a Siegel parabolic subgroup whose Levi subgroup is isomorphic to $\GL_r$. 
Given a complex number $s$, we have a normalized induced representation $\rho_{\tau,s}$ of 
$H_r(F)$ with the underlying space $I_r(\tau,s)$, which consists of smooth functions $\xi_s: H_r(F)\to\cV_\tau$ satisfying 
the usual rule (cf. \S\ref{SS:ind rep}). In this work, we always assume that $1\le r\le n$. 
Then $H_r(F)$ can be regarded as a subgroup of $G_n(F)$ via a natural embedding (cf. \eqref{E:embedding}) and we 
have the Rankin-Selberg integrals $\Psi_{n,r}(v\ot\xi_s)$ attached to $v\in\cV_\pi$ and $\xi_s\in I_r(\tau,s)$. 
(cf. \S\ref{SSS:RS integral and FE}). 

\subsection{Main results}\label{SS:main}
Let $\tau$ be an unramified representation of $\GL_r(F)$ that is induced 
of Langlands' type (cf. \S\ref{SS:Langlands type}). This includes all (classes of) irreducible smooth generic unramified 
representations of $\GL_r(F)$ and the space $\cV^{\GL_r(\frak{o})}_\tau$ is one-dimensional. 
Moreover, $\tau$ admits a unique (up to scalars) nonzero Whittaker functional $\Lambda_{\tau,\b{\psi}}$ 
which is nontrivial on $\cV_{\tau}^{\GL_r(\frak{o})}$ (\cite[Section 1]{Jacquet2012}). We fix an element 
$v_\tau\in\cV_\tau^{\GL_r(\frak{o})}$ with
\[
\Lambda_{\tau,\b{\psi}}(v_\tau)=1.
\]
Let $J(\tau)$ be the unique irreducible unramified quotient of $\tau$ (\cite[Corollary 1.2]{Matringe2013}) and 
$\phi_{J(\tau)}$ be the $L$-parameter of $J(\tau)$ under the local Langlands correspondence for $\GL_r$ 
(\cite{HarrisTaylor2001}, \cite{Henniart2000}, \cite{Scholze2013}).
Put $R_{r,m}=K_{n,m}\cap H_r(F)$ for $m\ge 0$. Then the properties of $R_{r,m}$ (cf. \S\ref{SS:R_r,m}) imply that 
the space $I_r(\tau,s)^{R_{r,m}}$ is one-dimensional which admits a generator $\xi^m_{\tau,s}$ with 
$\xi^m_{\tau,s}(I_{2r})=v_\tau$.\\

We impose the following hypothesis on $\pi$, which holds when (i) $n=1,2$ (any $\pi$) and (ii) $\pi$ is supercuspidal or 
unramified (any $n$).

\begin{hypA}\label{H}
The space $\cV_\pi^{K_{n,a_\pi}}$ is one-dimensional and $\Lambda_{\pi,\psi}$ is nontrivial on $\cV_{\pi}^{K_{n,a_\pi}}$.
\end{hypA}

Now our first result can be stated as follows.

\begin{thm}\label{T:main}
Under the Hypothesis \ref{H}, we have 
\begin{equation}\label{E:dim bd}
{\rm dim}_\bbC\cV_\pi^{K_{n,m}}
\ge
\begin{pmatrix}
n+\lfloor\frac{m-a_\pi}{2}\rfloor\\n
\end{pmatrix}
+
\begin{pmatrix}
n+\lfloor\frac{m-a_\pi+1}{2}\rfloor-1\\n
\end{pmatrix}
\end{equation}
for $m>a_\pi$ and the action of the quotient group $J_{n,a_\pi}/K_{n,a_\pi}$ on $\cV_\pi^{K_{n,a_\pi}}$ gives $\e_\pi$.
Moreover, if $v_\pi$ is an element in $\cV_\pi^{K_{n,a_\pi}}$ with $\Lambda_{\pi,\psi}(v_\pi)=1$, then we have
\footnote{After this work is completed, David Loeffler informed the author that he had obtained a similar identity (in 
his unpublished note) by computing the Novodvorsky's local integrals (\cite{Novodvorsky1979}) for 
${\rm GSp}_4\x\GL_2$ attached to newforms (of ${\rm GSp}_4$).}
\begin{equation}\label{E:main eqn}
\Psi_{n,r}(v_\pi\ot\xi^{a_\pi}_{\tau,s})
=
\frac{L(s,\phi_\pi\ot\phi_{J(\tau)})}{L(2s,\phi_{J(\tau)},\bigwedge^2)}
\end{equation}
provided that the Haar measures are properly chosen (cf. \S\ref{SSS:Haar}).
\end{thm}

\begin{remark}
In \cite{Tsai2013}, Tsai proved that \eqref{E:main eqn} holds when $\pi$ is supercuspidal and $r=1,n$. Note that 
in these cases, $L(s,\phi_\pi\ot\phi_{J(\tau)})=1$. On the other hand, when $\pi$ is unramifed, the identity
\eqref{E:main eqn} was first obtained by Gelbart and Piatetski-Shapiro (\cite[Proposition A.1]{GPSR1987}) for 
$r=n$, and then by Ginzburg (\cite[Theorem B]{Ginzburg1990}) for $r<n$.
\end{remark}

To prove \eqref{E:dim bd}, we define subsets $\cB_{\pi,m}\subset\cV_\pi^{K_{n,m}}$ (similar to Tsai's) in 
\S\ref{SS:conj basis} and show that under the Hypothesis \ref{H}, $\cB_{\pi,m}$ are linearly independent and have the 
cardinality given by the RHS of \eqref{E:dim formula} (and hence verifies \eqref{E:dim bd}). 
In particular, if Conjecture \ref{C1} $(1)$ holds, then $\cB_{\pi,m}$ define bases of 
$\cV_\pi^{K_{n,m}}$ for $m>a_\pi$. In this work, we also compute the integrals $\Psi_{n,r}(v\ot\xi_{\tau,s}^m)$ for 
$v\in\cB_{\pi,m}$ for $1\le r\le n$ (cf. \thmref{T:main for oldform}). It should be indicated that although we only compute the 
Rankin-Selberg integrals when $\tau$ is unramified, it turns out to be enough. In fact, we will show (without Hypothesis 
\ref{H}) that if $v$ is a nonzero element in $\cV_\pi^{K_{n, m}}$ (for some $m$) and $\tau$ is ramified, then 
$\Psi_{n,r}(v\ot\xi_s)=0$ for all $\xi_s\in I_r(\tau,s)$ for $1\le r\le n$ (cf. \lmref{L:vanish of RS integral}).\\ 

When $n=2$, we have other bases of oldforms defined by \eqref{E:oldform}. Alternatively, one can also compute 
the Rankin-Selberg integrals attached to these bases. When $r=1$, these are given by \eqref{E:zeta oldform n=2}.
Here we have the following result when $r=2$ (and $\tau$ unramified).

\begin{thm}\label{T:main'}
Let $v\in\cV_\pi^{K_{2,m}}$. We have
\begin{equation}\label{E:raising theta}
\Psi_{2,2}(\theta(v)\ot\xi^{m+1}_{\tau,s})=q^{-s+\frac{3}{2}}(\alpha_1+\alpha_2)\Psi_{2,2}(v\ot\xi^m_{\tau,s})
\end{equation}
and
\begin{equation}\label{E:raising theta'}
\Psi_{2,2}(\theta'(v)\ot\xi^{m+1}_{\tau,s})=q(1+q^{-2s+1}\alpha_1\alpha_2)\Psi_{2,2}(v\ot\xi^m_{\tau,s})
\end{equation}
and
\begin{equation}\label{E:raising eta}
\Psi_{2,2}(\eta(v)\ot\xi^{m+2}_{\tau,s})=q^{-2s+2}\alpha_1\alpha_2\Psi_{2,2}(v\ot\xi^m_{\tau,s})
\end{equation}
where $\alpha_1$ and $\alpha_2$ are the Satake parameters of $J(\tau)$.
\end{thm}

\begin{remark}
Combining \eqref{E:main eqn} with \thmref{T:main'}, one can explicitly compute the integrals 
$\Psi_{2,2}(v\ot\xi^m_{\tau,s})$ for every $v\in\cV_\pi^{K_{2,m}}$ and $m\ge 0$. 
\end{remark}

\subsection{An outline of the paper}
We end this introduction by a brief outline of the paper. In \S\ref{S:newform}, we will introduce the paramodular 
subgroups of $\SO_{2n+1}$ defined by Gross and Tsai. In \S\ref{S:oldform}, we will describe conjectural bases for 
oldforms given implicitly by Tsai in her thesis. In \S\ref{S:RS integral}, we will introduce the Rankin-Selberg integrals 
attached to generic representations of $\SO_{2n+1}\x\GL_r$ (for $1\le r\le n$). In \S\ref{S:unramified rep}, we will collect 
some facts for unramified representations including formulae of the spherical Whittaker functions of $\GL_r$ and various 
Satake isomorphisms. The goal of \S\ref{S:key} is to prove \propref{P:main prop}, which is the core of this work. 
The idea is to apply the standard techniques developed by Jacquet et al. to our setting. We also invoke an idea of 
Ginzburg, which allows us to obtain the results for $r<n$ from that for $r=n$. In \S\ref{S:prove}, we will prove our main 
results. We end with a comparison between the bases of oldforms given by Roberts-Schmidt, Casselman and Tsai.

\subsection{Notation and conventions}
Let $F$ be a finite extension of $\bbQ_p$, $\frak{o}$ be the valuation ring of $F$, $\frak{p}$ be the maximal ideal of 
$\frak{o}$, $\varpi\in\frak{p}$ be a prime element and $q$ be the cardinality of the residue field $\frak{f}=\frak{o}/\frak{p}$. 
We fix an additive character $\psi$ of $F$ with $\ker(\psi)=\frak{o}$. Let $\nu_1=|\cdot|_F$ be the absolute value on $F$ 
normalized so that $|\varpi|_F=q^{-1}$. In general, if $r\ge 1$ is an integer, we let $\nu_r$ be the character on $\GL_r(F)$ 
defined by $\nu_r(a)=|\det(a)|_F$. Suppose that $G$ is an $\ell$-group in the sense of \cite[Section 1]{BZ1976}. 
We denote by $\delta_G$ the modular function of $G$. If $K\subset G$ is an open compact subgroup, then $\cH(G//K)$ 
denotes the algebra of locally constant, compact supported $\bbC$-valued functions on $G$ which are bi-$K$-invariant. 
In this work, by a representation of $G$ we mean a smooth representation with coefficients in $\bbC$. 
If $\pi$ is a representation of $G$, then its underlying (abstract) space is usually denoted by $\cV_\pi$. Finally, we define  
elements $\jmath_r\in\GL_r(F)$ inductively by 
\begin{equation}\label{E:J_m}
\jmath_1=1\quad\text{and}\quad \jmath_r=\pMX{0}{1}{\jmath_{r-1}}{0}
\end{equation}
and write $I_r$ for the identity matrix in $\GL_r(F)$.

\section{Paramodular subgroups of $\SO_{2n+1}$}\label{S:newform}
In this section, we introduce the open compact subgroups $K_{n,m}$ of $\SO_{2n+1}$ defined by Gross 
(\cite[Section 5]{Gross2015}) and Tsai (\cite[Chapter 7 ]{Tsai2013}). Following Roberts-Schmidt, we will call $K_{n,m}$
the $paramodular$ $subgroups$. 
 
\subsection{Odd special orthogonal groups}\label{SSS:SO_{2n+1}}
The group $G_n=\SO_{2n+1}$ is the special orthogonal group of the quadratic space $(V_n,(,))$ with
\[
V_n=F e_1\oplus\cdots\oplus Fe_n\oplus Fv_0\oplus Ff_n\oplus\cdots\oplus Ff_1
\] 
and the symmetric bilinear form $(,):V_n\x V_n\to F$ given by $(v_0,v_0)=2$, $(e_i,f_i)=1$ for $1\leq i\leq n$ and all other
inner products are zero. Thus the Gram matrix of the ordered basis $\stt{e_1,\ldots, e_n, v_0, f_n,\ldots, f_1}$ is 
\begin{equation}\label{E:S}
S
=
\begin{pmatrix}
&&\jmath_n\\
&2&\\
\jmath_n
\end{pmatrix}\in\GL_{2n+1}(F).
\end{equation}
We have
\[
G_n(F)=\stt{g\in{\rm SL}_{2n+1}(F)\mid \,^{t}gSg=S}
\]
where $^{t}g$ denotes the transpose of $g$. Let $U_n\subset G_n$ be the upper triangular maximal unipotent subgroup
and define the non-degenerate character $\psi_{U_n}: U_n(F)\to\bbC^\x$ by 
\begin{equation}\label{E:psi_U}
\psi_{U_n}(u)=\psi(u_{12}+u_{23}+\cdots+u_{n-1, n}+ 2^{-1}u_{n,n+1})
\end{equation}
for $u=(u_{ij})\in U_n(F)$. 

\subsection{Paramodular subgroups}\label{SS:paramodular subgroup}
Let $\bbL_m\subset V_n$ be the $\frak{o}$-lattice defined by
\[
\bbL_{m}
=
\frak{o}e_1\oplus\cdots\oplus\frak{o}e_n\oplus\frak{p}^m v_0\oplus\frak{p}^mf_n\oplus\cdots\oplus\frak{p}^mf_1
\]
for $m\ge 0$. Let $J_{n,m}\subset G_n(F)$ be the open compact subgroup stabilizing the lattice $\bbL_m$,
so that $J_{n,0}=G_n(\frak{o})$ is the hyperspecial maximal compact subgroup of $G_n(F)$. 
We indicate that the explicit shape of $J_{n,m}$ can be found in \cite[Section 3.4]{Shahabi2018}. For $m\geq 1$,  
$J_{n,m}$ is isomorphic to the group of $\frak{o}$ points of a group scheme over $\frak{o}$. More precisely, if we equip 
$\bbL_{m}$ with the symmetry bilinear form $(,)_m:=\varpi^{-m}(,)$, then the Gram matrix of $(\bbL_{m}, (,)_m)$ 
associated to the ordered basis $\stt{e_1,\ldots, e_n, \varpi^m v_0, \varpi^m f_n,\ldots, \varpi^mf_1}$ is 
\[
S_m
=
\begin{pmatrix}
&&\jmath_n\\
&2\varpi^m&\\
\jmath_n&&
\end{pmatrix}
\]
and 
\[
\tilde{J}_{n,m}:=
\stt{g=(g_{ij})\in{\rm SL}_{2n+1}(\frak{o})\mid
\text{$^tgS_m g=S_m$ and $g_{j,n+1}\in\frak{p}^m$ for $1\leq j\neq n+1\leq 2n+1$}}
\]
is the group of $\frak{o}$ point of a group scheme $\tilde{\bf{J}}_{n,m}$ over $\frak{o}$, i.e. 
$\tilde{J}_{n,m}=\tilde{\bf{J}}_{n,m}(\frak{o})$ (cf. \cite[Theorem 3.6]{Shahabi2018}). Now we have 
$J_{n,m}=t_m^{-1}\tilde{J}_{n,m}t_{m}$ (in $\GL_{2n+1}(F)$), where
\[
t_m
=
\begin{pmatrix}
\varpi^m I_n&&\\
&1&\\
&&I_n
\end{pmatrix}.
\]
We define $K_{n,0}=J_{n,0}=G_n(\frak{o})$. For $m\geq 1$, on the other hand, the reduction modulo $\frak{p}$ map 
gives rise to a surjective homomorphism
\[
\tilde{J}_{n,m}=\tilde{\bf{J}}_{n,m}(\frak{o})
\relbar\joinrel\twoheadrightarrow
{\rm S}({\rm O}_{2n}(\frak{f})\x{\rm O}_1(\frak{f}))
\relbar\joinrel\twoheadrightarrow
{\rm O}_{2n}(\frak{f})
\overset{{\rm det}}{\relbar\joinrel\twoheadrightarrow}
\stt{\pm 1}.
\]
Then $K_{n,m}\subset J_{n,m}$ is defined to be the index $2$ normal subgroup such that 
$t_m K_{n,m} t_m^{-1}\subset \tilde{J}_{n,m}$ is the kernel of the above homomorphism. 
Again, the explicit shape of $K_{n,m}$ can be found in \cite[Section 3.5]{Shahabi2018}. 

\begin{remark}\label{R:paramodular subgroup}
As pointed out in \cite{Gross2015}, the definition of $K_{n,m}$ was suggested by Brumer, based on his extensive 
computations in \cite{BrumerKenneth2014}. 
\end{remark}

After defining the paramodular subgroups, we now prove:

\begin{lm}\label{L:conj for n=2}
Conjecture \ref{C1} holds when $n=1,2$.
\end{lm}

\begin{proof}
Since $K_{1,m}\cong \Gamma_0(\frak{p}^m)$ and $K_{2,m}\cong K(\frak{p}^m)$ under the accidental isomorphisms 
$\SO_3\cong{\rm PGL}_2$ and $\SO_5\cong{\rm PGSp}_4$, we can apply relevant results for generic representations 
of ${\rm PGL}_2(F)$ and ${\rm PGSp}_4(F)$. When $n=1$, lemma follows from the results in \cite{Casselman1973}
and \cite{Schmidt2002}. When $n=2$, we can apply the results of Roberts-Schmidt in \cite{RobertsSchmidt2007}. 
However, one thing needs to be clarified. Let $\pi$ be an irreducible generic representation of 
$\SO_5(F)\cong {\rm PGSp}_4(F)$. We denote by $L(s,\pi)$, $\epsilon(s,\pi,\psi)$ and $\gamma(s,\pi,\psi)$ the $L$-, 
$\epsilon$- and $\gamma$-factor attached to $\pi$ and $\psi$ defined by Novodvosky's zeta integrals for 
${\rm GSp}_4\x\GL_1$ (\cite[Section 2.6]{RobertsSchmidt2007}). Then as mentioned in the introduction, Roberts-Schmidt 
proved that there exists an integer $N_\pi\ge 0$ such that $\cV_\pi^{K_{2,m}}=0$ if $0\le m< N_\pi$ and 
${\rm dim}_\bbC\cV_\pi^{K_{2,N_\pi}}=1$. They also defined Atkin-Lehner elements $u_m\in{\rm GSp}_4(F)$ for 
$m\ge 0$ (\cite[Page 5]{RobertsSchmidt2007}) and showed that if the action of $\pi(u_{N_\pi})$ on the line 
$\cV_\pi^{K_{2,N_\pi}}$ gives $\e'_\pi\in\stt{\pm 1}$, then 
\[
\epsilon(s,\pi,\psi)=\e_\pi' q^{-N_\pi(s-\frac{1}{2})}.
\]
Under the isomorphism $\SO_5\cong{\rm GSp}_4$, it's not hard to see that $u_0\in J_{2,0}=K_{2,0}$ and
$u_m\in J_{2,m}\setminus K_{2,m}$ if $m>0$ (\cite[Section 6.2]{Tsai2013}).  To apply their results (comparing with 
Conjecture \ref{C1}), we have to show
\[
\epsilon(s,\pi,\psi)=\epsilon(s,\phi_\pi,\psi).
\]
When $\pi$ is not supercuspidal, this was already verified by Roberts-Schmidt. So let's assume that $\pi$ is supercuspidal.
Then we note that Novodvorsky's zeta integrals (in \cite{RobertsSchmidt2007}) are equal to the Rankin-Selberg 
integrals for $\SO_5\x\GL_1$ (with $\tau$ being trivial ) under the accidental isomorphism (cf. Remark 
\ref{R:zeta integral}). Hence we can apply the results of Soudry, Jiang-Soudry and Kaplan
(cf. \thmref{T:RS=Gal gamma}) to obtain
\[
\gamma(s,\pi,\psi)=\gamma(s,\phi_\pi,\psi).
\]
Since $\pi$ is supercuspidal, we have $L(s,\pi)=L(s,\phi_\pi)=1$ by \cite[Theorem 1.4]{JiangSoudry2003} and 
\cite[Proposition 3.9]{Takloo-Bighash2000}. Now the desired identity between $\epsilon$-factors follows.  
\end{proof}

\section{Conjectural Basis for Oldforms}\label{S:oldform}
Let $\pi$ be an irreducible generic representation of $G_n(F)$.
In this section, we will define the subsets $\cB_{\pi,m}$ of $\cV_\pi^{K_{n,m}}$ for $m>a_\pi$ mentioned in the 
introduction. They are built on (conjectural) newforms and certain level raising operators coming from elements in the
spherical Hecke algebras of $\SO_{2n}$. These subsets were already appeared implicitly in Tsai's thesis
(\cite[Propositions 8.1.5, 9.1.6, 9.1.7]{Tsai2013}). However, there are some inaccuracies in her formulation (when $\pi$ is 
supercuspidal) and will be fixed in this paper.

\subsection{Even special orthogonal groups}\label{SS:SO even}
The group $H_r(F)=\SO_{2r}(F)$ can be realized as the subgroup of ${\rm SL}_{2r}(F)$ via 
\[
H_r(F)
=
\stt{h\in{\rm SL}_{2r}(F)\mid {}^th\jmath_{2r}h=\jmath_{2r}}
\]
and can be embedded into $G_n(F)$ as a closed subgroup fixing the vectors
$e_{r+1}, \cdots, e_n, v_0, f_n,\cdots, f_{r+1}$ pointwisely. In coordinates, we have
\begin{equation}\label{E:embedding}
H_r(F)\ni\pMX{a}{b}{c}{d}\longmapsto
\begin{pmatrix}
a&&b\\
&I_{2(n-r)+1}\\
c&&d
\end{pmatrix}\in G_n(F)
\end{equation}
for some $a,b,c,d\in{\rm Mat}_{r\x r}(F)$. 
In the followings, we do not distinguish $H_r(F)$ with its image in $G_n(F)$ under this embedding. 
This shall not cause serious confusions.

\subsection{Dominant weights}
Let $T_n\subset G_n$ be the maximal split diagonal torus and put $T_r=T_n\cap H_r$ for $1\le r\le n$.
Let $\epsilon_1, \epsilon_2,\ldots, \epsilon_n$ be the standard basis of $X_\bullet(T_n)={\rm Hom}(\bbG_m, T_n)$ such 
that $\epsilon_j(y)\in T_n(F)$ is the diagonal matrix whose $(j,j)$-entry and $(2n+2-j,2n+2-j)$-entry is $y$ and $y^{-1}$
respectively, and all other diagonal entries are $1$, for $1\le j\le n$. We also write $y^\la=\la(y)$ for $y\in F^\x$ and 
$\la\in X_\bullet(T_n)$. Let $\|\cdot\|:X_\bullet(T_n)\to\bbZ_{\geq 0}$ be the sup-norm with respect 
to the standard basis and ${\rm tr}:X_\bullet(T_n)\to\bbZ$ be the trace map defined by
\begin{equation}\label{E:trace}
{\rm tr}(\lambda)=\la_1+\la_2+\cdots+\la_n
\quad\text{where}\quad
\lambda=\la_1\epsilon_1+\la_2\epsilon_2+\cdots+\la_n\epsilon_n.
\end{equation}
Let $P^+_{G_n}\subset P^+_{H_n}\subset X_\bullet(T_n)$ be the subsets given by
\[
P^+_{G_n}
=
\stt{\la_1\epsilon_1+\la_2\epsilon_2+\cdots+\la_n\epsilon_n\mid \la_1\ge \la_2\ge \cdots\ge \la_n\ge 0}
\subset
P^+_{H_n}
=
\stt{\la_1\epsilon_1+\la_2\epsilon_2+\cdots+\la_n\epsilon_n\mid \la_1\ge \la_2\ge \cdots\ge |\la_n|}.
\]
We put
\[
\mu_r=\epsilon_1+\epsilon_2+\cdots+\epsilon_r\in X_\bullet(T_n).
\]

\subsection{Filtrations of twisted paramodular subgroups}\label{SS:filtration of K}
As observed in \cite{Tsai2013}, there are filtrations between paramodular subgroups $K_{n,m}$ according to
the parity of $m$ after we conjugate $K_{n,m}$ by certain torus elements. More precisely, let 
\begin{equation}\label{E:K*}
K^{(m)}_{n,m'}
=
\varpi^{(\frac{m'-m}{2})\mu_n}\cdot K_{n,m'}\cdot\varpi^{-(\frac{m'-m}{2})\mu_n}
\end{equation}
for $m'\geq m$ with the same parity, then we have 
\begin{equation}\label{E:filtration}
K_{n,m}=K^{(m)}_{n,m}\supset K^{(m)}_{n,m+2}\supset\cdots\supset K^{(m)}_{n,m+2\ell}\supset\cdots\supset R_{n,m}
=\bigcap_{m'\geq m,\,m'\equiv m\pmod 2} K^{(m)}_{n,m'}.
\end{equation}
When $m=0,1$, \eqref{E:filtration} was verified in \cite[Section 7.1]{Tsai2013}. The general case is a consequence of that. 
From \eqref{E:filtration} we immediately get
\begin{equation}\label{E:R* fix space}
\cV_\pi^{R_{n,m}}=\bigcup_{m'\geq m,\,m'\equiv m\pmod 2}\cV_{\pi}^{K^{(m)}_{n,m'}}.
\end{equation}

\subsection{Properties of $R_{r,m}$}\label{SS:R_r,m}
We record some properties of $R_{r,m}$ in this subsection (cf. \cite[Section 2.5]{Tsai2013}). 
\begin{itemize}
\item[(i)]
$R_{r, 0}=H_r(\frak{o})$ and $R_{r, 1}$ are two non-conjugate hyperspecial maximal compact subgroups of $H_r(F)$ and 
$R_{r,m}\cap M_r(F)=M_r(\frak{o})$ for $m\ge 0$.
\item[(ii)]
The relation
\begin{equation}\label{E:R*}
R_{r,m'}=\varpi^{-(\frac{m'-m}{2})\mu_r}\cdot R_{r,m}\cdot\varpi^{(\frac{m'-m}{2})\mu_r}
\end{equation}
holds for $m,m'\geq 0$ with the same parity. This follows from \eqref{E:K*}, \eqref{E:filtration} and the fact that 
$R_{r,m}=R_{n,m}\cap H_r(F)$.
\item[(iii)]
By (i) and (ii), we have the Iwasawa decomposition
\begin{equation}\label{E:Iwasawa decomp}
H_r(F)=Q_r(F)R_{r,m}=B_{H_r}(F)R_{r,m}
\end{equation}
for $m\ge 0$, where $B_{H_r}$ denotes the upper triangular Borel subgroup of $H_r$.
\item[(iv)]
As subgroups in $\GL_{2r}(F)$, we have
\begin{equation}\label{E:R conj}
t_{r,m}^{-1}R_{r,m}t_{r,m}=R_{r,0}
\quad\text{and}\quad
w_{r,m}^{-1}R_{r,m}w_{r,m}=R_{r,m}
\end{equation}
where
\begin{equation}\label{E:t and w}
t_{r,m}=\pMX{\varpi^{-m} I_r}{}{}{I_r}
\quad\text{and}\quad
w_{r,m}
=
\begin{pmatrix}
&\varpi^{-m}I_r\\\varpi^mI_r
\end{pmatrix}.
\end{equation}
Note that $w_{r,m}\in R_{r,m}$ when $r$ is even and $w_{r,m}\in{\rm O}_{2r}(F)\setminus H_r(F)$ when $r$ is odd.
\end{itemize}

We then have the following consequence.

\begin{lm}\label{L:same vol}
Let $dh$ be an arbitrary Haar measure on $H_r(H)$. Then we have ${\rm vol}(R_{r,m}, dh)={\rm vol}(H_r(\frak{o}),dh)$ 
for all $m\geq 0$.
\end{lm}

\begin{proof}
Consider the similitude group $\t{H}_r={\rm GSO}_{2r}\subset\GL_{2r}$ defined by same matrix $\jmath_{2r}$ 
with the center $Z$. Then $\t{H}_r(F)$ contains $H_r(F)$ as a subgroup and $Z(F)\cong F^\x$. Moreover, one
has $Z(F)\cap H_r(F)=\stt{I_{2r}}$ and $Z(F)H_r(F)$ is a closed subgroup in $\t{H}_r(F)$ with finite index, and hence 
also open. Fix a Haar measure $dz$ on $Z(F)$. Then the Haar measure $dzdh$ on $Z(F)H_r(F)$ induces a unique 
Haar measure $d\t{h}$ on $\t{H}_r(F)$, whose restriction to $Z(F)H_r(F)$ is $dzdh$.
Now we note that $t_m\in\t{H}_r(F)$, so that the first conjugation in \eqref{E:R conj} is actually in 
$\t{H}_r(F)$. It then follows that
\[
{\rm vol}(Z(\frak{o}),dz){\rm vol}(R_{r,m},dh)
=
{\rm vol}(Z(\frak{o})R_{r,m}, d\t{h})
=
{\rm vol}(Z(\frak{o})R_{r,0}, d\t{h})
=
{\rm vol}(Z(\frak{o}),dz){\rm vol}(R_{r,0},dh).
\]
Since $H_r(\frak{o})=R_{r,0}$, the proof follows.
\end{proof}

\subsection{Level raising operators}\label{SSS:level raising}
%We introduce certain level raising operators, which will be used in the formulation of the conjectural basis for oldforms.
To define the level raising operators, we use elements in $\cH(H_n(F)//R_{n,m})$, whose standard basis can 
be described as follows. For $\lambda\in P^+_{H_n}$, let
\[
\varphi_{\lambda, m}=\bbI_{R_{n,m}\varpi^{\lambda}R_{n,m}}
\]
be the characteristic function of the double coset $R_{n,m}\varpi^\lambda R_{n,m}$. Then we have
\[
\cH(H_n(F)//R_{n,m})=\bigoplus_{\lambda\in P^+_{H_n}}\bbC\cdot\varphi_{\lambda, m}
\]
as $\bbC$-linear spaces. This follows from the cases $m=0, 1$ (cf. \cite[Page 51]{Tits1979}) together with \eqref{E:R*}.\\

Given $\varphi\in\cH(H_n(F))$ and $v\in\cV_\pi$, define $\varphi\star v\in\cV_\pi$ by
\begin{equation}\label{E:Hecke action for H}
\varphi\star v
=
\int_{H_n(F)}\varphi(h)\pi(h^{-1})vdh
\quad
\text{(${\rm vol}(H_n(\frak{o}),dh)=1$).}
\end{equation}
Now if $v\in\cV_\pi^{K_{n,m}}$, then we have $\varphi_{\lambda, m}\star v\in\cV_\pi^{R_{n,m}}$ and hence by 
\eqref{E:R* fix space}, $\varphi_{\lambda, m}\star v\in \cV_\pi^{K^{(m)}_{n,m'}}$ for some $m'\geq m$ with the same 
parity. The following lemma tells us what $m'$ is.

\begin{lm}\label{L:level raising for K*}
Let $v\in\cV_\pi^{K_{n,m}}$ and $\la\in P^+_{H_n}$ with $\|\lambda\|=\ell$. 
Then we have $\varphi_{\lambda,m}\star v\in\cV_\pi^{K^{(m)}_{n,m+2\ell}}$.
\end{lm}

\begin{proof}
Let $e\in\stt{0,1}$ and $m'\geq m$ with $m'\equiv m\equiv e\pmod{2}$. Then by \eqref{E:K*}, we have
\[
\pi(\varpi^{(\frac{m-e}{2})\mu_n})\cV_\pi^{K^{(m)}_{n,m'}}=\cV_\pi^{K^{(e)}_{n,m'}}
\]
and hence $v':=\pi(\varpi^{(\frac{m-e}{2})\mu_n})v\in\cV_\pi^{K^{(e)}_{n,m}}$. Now 
\cite[Proposition 8.1.1]{Tsai2013} implies 
\[
\varphi_{\lambda,e}\star v'\in\cV_\pi^{K^{(e)}_{n,m+2\ell}}=\pi(\varpi^{(\frac{m-e}{2})\mu_n})
\cV_\pi^{K^{(m)}_{n,m+2\ell}}
\]
so that
\[
\pi(-\varpi^{(\frac{m-e}{2})\mu_n})\varphi_{\lambda,e}\star v'\in\cV_\pi^{K^{(m)}_{n,m+2\ell}}.
\]
But by definition, we have
\begin{align*}
\pi(\varpi^{-(\frac{m-e}{2})\mu_n})\varphi_{\lambda,e}\star v'
&=
\int_{H_n(F)}
\varphi_{\lambda,e}(h)\pi(\varpi^{-(\frac{m-e}{2})\mu_n}h^{-1}\varpi^{(\frac{m-e}{2})\mu_n})vdh\\
&=
\int_{H_n(F)}\varphi_{\lambda,m}(h)\pi(h^{-1})vdh\\
&=
\varphi_{\lambda,m}\star v
\end{align*}
after changing variable and taking \eqref{E:R*} into account. This proves the lemma.
\end{proof}

Combining \eqref{E:K*}, \eqref{E:filtration} with \lmref{L:level raising for K*}, we obtain level raising operators
$\eta_{\lambda, m, m'}:\cV_\pi^{K_{n,m}}\to\cV_\pi^{K_{n,m'}}$ defined by
\[
\eta_{\lambda,m,m'}
=
\pi(\varpi^{-(\frac{m'-m}{2})\mu_n})\circ\varphi_{\lambda, m}
\]
for every $m'\geq m+2\|\la\|$ with $m'\equiv m\pmod{2}$. Note that $\eta_{0,m,m}$ is the identity operator by 
\lmref{L:same vol}. Note also that the operators $\eta_{\lambda,m,m'}$ raise levels to those with the same parity. We put 
\[
\eta=\pi(\varpi^{-\mu_n}).
\] 
Then $\eta$ induces an isomorphism between $\cV_\pi^{R_{n,m}}$ and $\cV_\pi^{R_{n,m+2}}$ (by \eqref{E:R*}) for every
$m\ge 0$. Moreover, the restriction of $\eta$ to $\cV_\pi^{K_{n,m}}$ gives the operator $\eta_{0,m,m+2}$. 
To raise levels to those with the opposite parity, we need two more operators 
$\theta_m, \theta'_m:\cV_\pi^{K_{n,m}}\to\cV_\pi^{K_{n,m+1}}$ defined by
\begin{equation}\label{E:theta}
\theta_m=\pi(u_{n,1,m+1})\circ\theta'_m\circ\pi(u_{n,1,m})
\quad
\text{and}
\quad
\theta'_m(v)=\frac{1}{{\rm vol}(K_{n,m}\cap K_{n,m+1}, dg)}\int_{K_{n,m+1}}\pi(g)v dg
\end{equation}
where $u_{n,1,m}\in J_{n,m}$ is the element given by \eqref{E:u_r,m}. In the followings, we usually suppress $m$ from 
the notations $\theta_m$ and $\theta'_m$ when there is no risk of confusion. 

\subsection{The subsets $\cB_{\pi,m}$}\label{SS:conj basis}
Now we are ready to introduce the subsets $\cB_{\pi,m}$. For this, we need one more notation:
If $\la=\la_1\epsilon_1+\la_{2}\epsilon_{2}+\cdots+\la_n\epsilon_n\in X_\bullet(T_n)$, then we write 
$\t{\la}=\la_1\epsilon_1+\la_2\epsilon_2+\cdots+\la_{n-1}\epsilon_{n-1}-\la_n\epsilon_n\in X_\bullet(T_n)$.
Let $v_\pi\in\cV_\pi^{K_{n,a_\pi}}$ be a (conjectural) newform and $m>a_\pi$. Then if $m$ and $a_\pi$ have the same 
parity, we put
\[
\cB_{\pi,m}=\stt{\eta_{\lambda,a_\pi,m}(v_\pi)\mid \lambda\in P^+_{H_n},\,2\|\lambda\|\leq m-a_\pi}.
\]
On the other hand, if $m$ and $a_\pi$ have the opposite parity, then the subsets are given by
\[
\cB_{\pi,m}=\stt{\eta^\square_{\la,a_\pi+1,m}\circ\theta(v_\pi),\,\,\eta^\square_{\lambda, a_\pi+1,m}\circ\theta'(v_\pi)\mid 
\lambda\in P^+_{G_n},\,\,2\|\lambda\|\leq m-a_\pi-1}
\]
where $\eta^\square_{\la,a_\pi+1,m}:=\eta_{\la,a_\pi+1,m}+\eta_{\t{\la}, a_\pi+1, m}$.\\

We have some remarks on $\cB_{\pi,m}$.
When $\pi$ is supercuspidal (any $n$) and $m$, $a_\pi$ have the same parity, the subsets 
$\cB_{\pi,m}$ were appeared implicitly in \cite[Propositions 8.1.6, 9.1.6]{Tsai2013}. When $m$ and $a_\pi$ have the 
opposite parity, on the other hand, Tsai asserted in \cite[Proposition 9.1.7]{Tsai2013} that the subsets 
$\cB'_{\pi,m}\subset\cV_\pi^{K_{n,m}}$ given by 
\[
\cB'_{\pi,m}=\stt{\eta_{\la,a_\pi+1,m}\circ\theta(v_\pi),\,\,\eta_{\lambda, a_\pi+1,m}\circ\theta'(v_\pi)\mid 
\lambda\in P^+_{G_n},\,\,2\|\lambda\|\leq m-a_\pi-1}
\]
are linearly independent. However, this is not quite correct (cf. \S\ref{SSS:concluding remark}). 
When $n=1,2$, the subsets $\cB_{\pi,m}$ give other bases for the spaces of oldforms. 
It's then natural to compare them with the ones given by Casselman and Roberts-Schmidt. 
This will also appear in \S\ref{SSS:concluding remark}.

\section{Local Rankin-Selberg integrals for $\SO_{2n+1}\x\GL_r$}\label{S:RS integral}
The global Rankin-Selberg integrals for generic cuspidal automorphic representations of $\SO_{2n+1}\x\GL_r$ were 
first introduced by Gelbart and Piatetski-Shapiro in \cite[Part B]{GPSR1987} when $r=n$. Their constructions were
later extended by Ginzburg in \cite{Ginzburg1990} to $r<n$. The constructions were 
further extended by Soudry (\cite{Soudry1993}) to $r>n$, who also developed the corresponding local theory.
In this section, we review the local Rankin-Selberg integrals for generic representations of 
$\SO_{2n+1}\x\GL_r$ with $1\leq r\leq n$ following \cite{Soudry1993}, see also \cite{Kaplan2015}.\\

Let $Z_r\subset\GL_r$ be the upper triangular maximal unipotent subgroup. 
Define a non-degenerate character of $Z_r(F)$ by 
\begin{equation}\label{E:psi_Z}
\b{\psi}_{Z_r}(z)
=
\ol{\psi(z_{12}+z_{23}+\cdots+z_{r-1, r})}
\end{equation}
for $z=(z_{ij})\in Z_r(F)$. Let $\tau$ be a representation of $\GL_r(F)$. We assume that $\tau$ has finite length 
and the $\bbC$-linear space ${\rm Hom}_{Z_r(F)}(\tau,\b{\psi}_{Z_r})$ is one-dimensional. 
We fix a nonzero element $\Lambda_{\tau,\b{\psi}}\in{\rm Hom}_{Z_r(F)}(\tau,\b{\psi}_{Z_r})$. Define $\tau^*$
to be the representation of $\GL_r(F)$ on $\cV_\tau$ with the action $\tau^*(a)=\tau(a^*)$, 
where $a^*=\jmath_{r}{}^t a^{-1}\jmath_r$. Then $\tau^*$ also has finite length and the $\bbC$-linear space 
${\rm Hom}_{Z_r(F)}(\tau^*,\b{\psi})$ is also one-dimensional. We fix a nonzero element 
$\Lambda_{\tau^*,\b{\psi}}\in {\rm Hom}_{Z_r(F)}(\tau^*,\b{\psi}_{Z_r})$ given by
$\Lambda_{\tau^*,\b{\psi}}=\Lambda_{\tau,\b{\psi}}\circ\tau(d_r),$ where 
\begin{equation}\label{E:eta_r}
d_r
=
\begin{pmatrix}
1&&&\\
&-1&&\\
&&\ddots\\
&&&(-1)^{r-1}
\end{pmatrix}
\in\GL_r(\frak{o}).
\end{equation}

\subsection{Induced representations and intertwining maps}\label{SS:ind rep}
Let $s$ be a complex number and $\tau_s$ be a representation of $\GL_r(F)$ on $\cV_\tau$ with the action 
$\tau_s(a)=\tau(a)\nu_r(a)^{s-\frac{1}{2}}$. The representation $\tau^*_{1-s}$ is defined in the similar way.

\subsubsection{Siegel parabolic subgroups}
Let $Q_r\subset H_r$ be a Siegel parabolic subgroup with the Levi decomposition $M_r\ltimes N_r$, where
\begin{equation}\label{E:levi of Q}
M_r(F)
=
\stt{m_r(a)=\pMX{a}{}{}{a^*}\mid \text{$a\in\GL_r(F)$ and $a^*=\jmath_r{}^t a^{-1}\jmath_r$}}\cong\GL_r(F)
\end{equation}
and 
\begin{equation}\label{E:unipotent radical of Q}
N_r(F)
=
\stt{n_r(b)=\pMX{I_r}{b}{}{I_r}\mid\text{$b\in\Mat_{r\x r}(F)$ with $b=-\jmath_r{}^t b\jmath_r$}}.
\end{equation}
There is another Siegel parabolic subgroup $\t{Q}_r\subset H_r$ obtained from $Q_r$ by conjugating a Weyl element 
\begin{equation}\label{E:delta}
\delta_r
=
\begin{pmatrix}
I_{r-1}&&\\
&\jmath_2&\\
&&I_{r-1}
\end{pmatrix}^r\in {\rm O}_{2r}(F)
\end{equation}
i.e. $\t{Q}_r(F)=\delta_r^{-1}Q_r(F){\delta_r}$.
It has the Levi decomposition $\t{Q}_r=\t{M}_r\ltimes\t{N}_r$ with $\t{M}_r(F)=\delta_r^{-1}M_r(F)\delta_r$ and 
$\t{N}_r(F)=\delta_r^{-1}N_r(F)\delta_r$ (cf. \cite[Chapter 9]{Soudry1993}). 

\subsubsection{Induced representations}\label{SSS:ind rep}
By pulling back the homomorphism $Q_r(F)\twoheadrightarrow Q_r(F)/N_r(F)\cong\GL_r(F)$, $\tau_s$ becomes a 
representation of $Q_r(F)$ on $\cV_\tau$. We then form a normalized induced representation
\[
\rho_{\tau, s}={\rm Ind}_{Q_r(F)}^{H_r(F)}(\tau_s)
\]
of $H_r(F)$. The representation space $I_r(\tau,s)$ of $\rho_{\tau,s}$ consists of smooth functions 
$\xi_s: H_r(F)\to\cV_\tau$ satisfying 
\begin{equation}\label{E:xi_s}
\xi_s(mnh)=\delta_{Q_r}^{\frac{1}{2}}(m)\tau_s(m)\xi_s(h)
\end{equation}
for $m\in M_r(F)$, $n\in N_r(F)$ and $h\in H_r(F)$. The action on $I_r(\tau,s)$ is given by the right translation $\rho$.
We remind here that $\delta_{Q_r}(m_r(a))=\nu_r(a)^{r-1}$.\\

Similarly, $\tau^*_{1-s}$ can be extended to a representation of $\t{Q}_r(F)$ through the homomorphism 
$\t{Q}_r(F)\twoheadrightarrow \t{Q}_r(F)/\t{N}_r(F)\cong\GL_r(F)$. Explicitly, its action on $\cV_\tau$ is given by 
\[
\tau^*_{1-s}(\t{m}_r(a)\t{n})
=
\tau^*(a)\nu_r(a)^{\frac{1}{2}-s}
\]
for $\t{m}_r(a)=\delta_r^{-1}m_r(a)\delta_r\in\t{M}_r(F)$ and $\t{n}\in\t{N}_r(F)$. 
We thus obtain another normalized induced representation 
\[
\t{\rho}_{\tau^*,1-s}
=
{\rm Ind}_{\t{Q}_r(F)}^{H_r(F)}(\tau^*_{1-s})
\]
of $H_r(F)$. Its underlying space $\t{I}_r(\tau^*,1-s)$ consists of smooth functions $\t{\xi}_{1-s}: H_r(F)\to\cV_\tau$ 
satisfying the rule similar to that of \eqref{E:xi_s}.

\subsubsection{Intertwining maps}\label{SSS:intertwining map}
There is an intertwining map between the induced representations $\rho_{\tau,s}$ and $\rho_{\tau^*,1-s}$. 
To define it, let $w_r\in H_r(F)$ be given by
\[
w_r
=
\begin{cases}
\begin{pmatrix}&I_r\\I_r&\end{pmatrix}\quad&\text{if $r$ is even},\\
\begin{pmatrix}&I_r\\I_r&\end{pmatrix}\begin{pmatrix}&&1\\&I_{2r-2}&\\1&&\end{pmatrix}\quad&\text{if $r$ is odd}.
\end{cases}
\]
Then the intertwining map $M(\tau, s):I_r(\tau,s)\to\t{I}_r(\tau^*, 1-s)$ is defined by the following integral
\begin{equation}\label{E:intertwining map}
M(\tau,s)\xi_s(h)
=
\int_{\t{N}_r(F)}\xi_s(w_r^{-1}\t{n}h)d\t{n}
\end{equation}
for $\Re(s)\gg 0$, and by meromorphic continuation in general. The Haar measure $d\t{n}$ on $\t{N}_r(F)$ is chosen 
as follows (\cite[Page 398]{Kaplan2015}). The group $\t{N}_r(F)$ can be written as a product of its root groups, each 
of which is isomorphic to $F$. We take the Haar measure on $F$ to be self-dual with respect to $\psi$, and then 
transport it onto each root group. Since $\ker(\psi)=\frak{o}$, the Haar measure on $F$ is such that the total 
volume of $\frak{o}$ is $1$. Now $d\t{n}$ is taken to be the product measure.\\ 

The assumptions on $\tau$ allow us to define Shahidi's local coefficient  
$\gamma(s,\tau,{\bigwedge}^2,\psi)\in\bbC(q^{-s})$ through his functional equation 
(\cite[Section 10]{Soudry1993}, \cite[Section 3.1]{Kaplan2015}). Then we have the normalized 
intertwining map
\[
M_\psi^\dagger(\tau, s)
=
\gamma(2s-1,\tau, {\Wedge}^2, \psi)M(\tau,s).
\]

\subsection{Rankin-Selberg integrals}\label{SSS:RS integral and FE}
Let $\pi$ be an irreducible generic representation of $G_n(F)$. Recall that we have fixed a nonzero Whittaker 
functional $\Lambda_{\pi,\psi}\in{\rm Hom}_{U_n(F)}(\pi,\psi_{U_n})$. Let
\begin{equation}\label{E:X_n,r}
\bar{X}_{n,r}(F)
=
\stt{
\begin{pmatrix}I_r&&&&\\x&I_{n-r}&&\\&&1\\&&&I_{n-r}&\\&&&x'&I_r\end{pmatrix}
\mid\text{$x\in{\rm Mat}_{(n-r)\x r}(F)$ with $x'=-\jmath_r{}^tx\jmath_{n-r}$}}
\end{equation}
be a unipotent subgroup of $G_n(F)$. Then the integral $\Psi_{n,r}(v\ot\xi_s)$ attached to $v\in\cV_\pi$ and 
$\xi_s\in I(\tau,s)$ is given by
\begin{equation}\label{E:RS integral in general}
\Psi_{n,r}(v\ot\xi_s)
=
\int_{V_r(F)\backslash H_r(F)}\int_{\b{X}_{n,r}(F)}
W_v(\b{x}h)f_{\xi_s}(h)d\b{x}dh
\end{equation}
where $W_v(g)=\Lambda_{\pi,\psi}(\pi(g)v)$ is the Whittaker function associated to $v$ and 
$f_{\xi_s}(h)=\Lambda_{\tau,\b{\psi}}(\xi_s(h))$ for $g\in G_n(F)$ and $h\in H_r(F)$.
On the other hand, by using the normalized intertwining operator $M_\psi^\dagger(\tau,s)$, we can define another 
integral $\t{\Psi}_{n,r}(v\ot\xi_s)$ attached to $v$ and $\xi_s$ as follows. Let
\begin{equation}\label{E:delta_n,r}
\delta_{n,r}
=
\begin{pmatrix}
I_{r-1}&&\\&&&1\\&&-I_{2(n-r)+1}\\&1\\&&&&I_{r-1}
\end{pmatrix}^r\in G_n(F).
\end{equation}
Observe the similarity between $\delta_r$ (cf. \eqref{E:delta}) and $\delta_{n,r}$. 
Indeed, $\delta_{n,r}$ is essentially obtained from $\delta_r$ via the embedding \eqref{E:embedding}. 
Then we have
\begin{equation}\label{E:dual RS integral in general}
\t{\Psi}_{n,r}(v\ot\xi_s)
=
\int_{V_r(F)\backslash H_r(F)}\int_{\b{X}_{n,r}(F)}
W_v(\b{x}h\delta_{n,r})f^*_{M^\dagger_\psi(\tau,s)\xi_s}(\delta_r^{-1}h\delta_r)d\b{x}dh
\end{equation}
where 
\[
f^*_{M^\dagger_\psi(\tau,s)\xi_s}(h)
=
\Lambda_{\tau^*,\b{\psi}}(M^\dagger_\psi(\tau,s)\xi_s(h)).
\]

As expected, these integrals, which are originally absolute convergence in some half planes, have meromorphic 
continuations to whole complex plane, and give rise to rational functions in $q^{-s}$. Moreover, the following
functional equations:
\begin{equation}\label{E:FE}
\t{\Psi}_{n,r}(v\ot\xi_s)=\gamma(s,\pi\x\tau,\psi)\Psi_{n,r}(v\ot\xi_s)
\end{equation}
hold for every $v\in\cV_\pi$ and $\xi_s\in I_r(\tau,s)$, where $\gamma(s,\pi\x\tau,\psi)$ is a nonzero rational function 
in $q^{-s}$ depending only on $\psi$ and (the classes of) $\pi$, $\tau$.

\begin{remark}\label{R:zeta integral}
When $r=1$, we have $H_1(F)=M_1(F)\cong F^\x$ and hence $\tau$ is a character of $F^\x$ and 
$\rho_{\tau,s}=\tau_s=\tau\nu_1^{s-\frac{1}{2}}$. It follows that $I_1(\tau,s)=\bbC\tau_s$ and 
$\t{I}_1(\tau^*,1-s)=\bbC\tau^{-1}_{1-s}$ are both one-dimensional. Let us put
\begin{equation}\label{E:zeta integral}
Z(s,v,\tau)
=
\int_{F^\x}\int_{\b{X}_{n,1}(F)}
W_v\left(\b{x}\epsilon_1(y)\right)\tau\nu_1^{s-\frac{1}{2}}(y)d\b{x}d^\x y. 
\end{equation}
for $v\in\cV_\pi$. Then $\Psi_{n,1}(v\ot\xi_s)=cZ(s,v,\tau)$ if $\xi_s=c\tau\nu_1^{s-\frac{1}{2}}$ for some $c\in\bbC$.
Note that $\delta_{n,1}=u_{n,1,0}\in G_n(F)$ (cf. \eqref{E:u_r,m}) and we have 
$\t{\Psi}_{n,1}(v\ot\xi_s)=cZ(1-s,\pi(u_{n,1,0})v, \tau^{-1})$. These are essentially the local integrals 
introduced by Jacquet-Langlands in \cite{JLbook} when $n=1$, and by Novodvorsky in \cite{Novodvorsky1979} when 
$n=2$. In the followings, if $\tau$ is the trivial character of $F^\x$, then we will drop $\tau$ from the notation.
\end{remark}

We record here the following result on the compatibility between $\gamma$-factors due to Soudry, Jiang-Soudry and 
Kaplan.

\begin{thm}[\cite{Soudry2000}, \cite{JiangSoudry2004}, \cite{Kaplan2015}]\label{T:RS=Gal gamma}
Let $\pi$ be an irreducible generic representation of $G_n(F)$ and $\tau$ be an irreducible generic 
representation of $\GL_r(F)$ for any $n,r\in\bbN$. Then we have 
\[
\gamma(s,\pi\x\tau,\psi)
=
\omega_\tau(-1)^n\gamma(s,\phi_\pi\ot\phi_\tau,\psi)
\]
where $\omega_\tau$ stands for the central character of $\tau$.
\end{thm}

\subsection{Two consequences} 
We derive two consequences from \eqref{E:filtration} and non-vanishing of Rankin-Selberg integrals. 
The first is existence of nonzero paramodular vectors, whose proof was given in \cite[Theorem 7.3.1]{Tsai2013} when 
$\pi$ is supercuspidal and was sketched in \cite[Proposition 7.6]{Tsai2016} for generic $\pi$. We provide a proof here 
for the sake of completeness. Recall the usual action
\[
\rho(\varphi)\xi_{s}
=
\int_{H_r(F)}\varphi(h)\rho(h)\xi_{s}dh
\]
for $\xi_s\in I(\tau,s)$ and $\varphi\in\cH(H_r(F))$.

\begin{lm}\label{L:existence}
We have $\cV_\pi^{K_{n,m}}\neq 0$ for all $m\gg 0$.
\end{lm}

\begin{proof}
Let $e\in\stt{0,1}$ such that  $m\equiv e\pmod 2$. Then we have $\cV_\pi^{K_{n,m}}\cong \cV_\pi^{K^{(e)}_{n,m}}$ by 
\eqref{E:K*} and hence it suffices prove the assertion for the spaces $\cV_\pi^{K^{(e)}_{n,m}}$. 
By \eqref{E:R* fix space}, we only need to verify that $\cV_\pi^{R_{n,e}}\neq 0$.
For this, we use the Rankin-Selberg integral for $\SO_{2n+1}\x\GL_n$. More precisely, let $\tau$ be an unramified
irreducible generic representation of $\GL_n(F)$, and pick $s_0\in\bbC$ so that 
\begin{itemize}
\item
$\rho_{\tau,s_0}$ is irreducible,
\item
$\Psi_{n,n}(v\ot\xi_{s_0})$ is absolutely convergent for all $v\in\cV_\pi$ and $\xi_{s_0}\in I_n(\tau, s_0)$.
\end{itemize}
By \cite[Proposition 12.4]{GPSR1987}, there exist $v\in\cV_\pi$ and $\xi_{s_0}\in I_n(\tau,s_0)$ such that 
$\Psi_{n,n}(v\ot\xi_{s_0})\ne 0$. On the other hand, since $I_n(\tau,s_0)$ is unramified and irreducible, the spaces 
$I_n(\tau,s_0)^{R_{n,e}}$ are both one-dimensional and $\xi_{s_0}$ can be written as $\rho(\varphi)\xi^e_{s_0}$ for some 
$\varphi\in\cH(H_n(F))$, where $\xi_{s_0}^e\in I_n(\tau, s_0)$ are nonzero spherical elements. So we may further assume
that $\varphi$ is right $R_{n,e}$-invariant. Now a simple computation shows
\[
\Psi_{n,n}(v\ot\xi_{s_0})=\Psi_{n,n}((\varphi\star v)\ot\xi^e_{s_0})\ne 0.
\]
Therefore we have $0\neq \varphi\star v\in\cV_\pi^{R_{n,e}}$.
\end{proof}

As another consequence, we have 

\begin{lm}\label{L:vanish of RS integral}
Let $v\in\cV_\pi^{K_{n,m}}$ be a nonzero element for some $m\geq 0$. Then the integrals $\Psi_{n,r}(v\ot\xi_s)$ vanish 
for all $\xi_s\in I_r(\tau,s)$ and $s\in\bbC$ for $1\le r\le n$ if $\tau$ is ramified.
\end{lm}  

\begin{proof}
The proof is similar to that of \lmref{L:existence}. First note that we have the following isomorphism  
\[
I_r(\tau,s)^{R_{r,m}}\cong\cV_{\tau_s}^{\GL_r(\frak{o})}
\]
(between $\bbC$-linear spaces) due to the Iwasawa decomposition $H_r(F)=Q_r(F)R_{r,m}$ and 
$R_{r,m}\cap M_r(F)\cong\GL_r(\frak{o})$. It follows that $I_r(\tau,s)^{R_{r,m}}$ is nonzero if and only if $\tau$ 
is unramified. Now suppose that $\tau$ is ramified and let $\varphi=\bbI_{R_{r,m}}\in\cH_r(F)$ be the characteristic 
function of $R_{r,m}$. Then we have $\varphi\star v=c_r v$ with $c_r={\rm vol}(R_{r,m},dh)$.
For $s_0\in\bbC$ with $\Re(s_0)\gg 0$ so that the integrals $\Psi_{n,r}(v\ot\xi_{s_0})$ converge absolutely for all 
$\xi_{s_0}\in I_r(\tau,s_0)$, one checks that
\[
c_r\Psi_{n,r}(v\ot\xi_{s_0})=\Psi_{n,r}((\varphi\star v)\ot\xi_{s_0})=\Psi_{n,r}(v\ot\rho(\varphi)\xi_{s_0}).
\]
Since $\rho(\varphi)\xi_{s_0}\in I_r(\tau,{s_0})^{R_{r,m}}=0$, we conclude that $\Psi_{n,r}(v\ot\xi_{s_0})=0$. For general 
$s_0\in\bbC$ and $\xi_{s_0}\in I_r(\tau,s_0)$, let $\xi'_s$ be the standard section such that $\xi'_{s_0}=\xi_{s_0}$. 
Then $\Psi_{n,r}(v\ot\xi_{s_0})$ is given by $\Psi_{n,r}(v\ot\xi'_s)|_{s=s_0}$ via the meromorphic continuation. 
But since $\Psi_{n,r}(v\ot\xi'_s)=0$ for all $\Re(s)\gg 0$ by what we have shown, the meromorphic function 
$\Psi_{n,r}(v\ot\xi'_s)$ is identically zero, and hence $\Psi_{n,r}(v\ot\xi_{s_0})=0$ as desired.  
\end{proof}

\section{Unramified representations}\label{S:unramified rep}
By \lmref{L:vanish of RS integral}, the computations of Rankin-Selberg integrals attached to nonzero paramodular vectors
reduce to the case when $\tau$ is unramified. Instead of assuming that $\tau$ is irreducible, we work with $\tau$ that
is induced of "Langlands' type". Another term to say is that $\tau$ is a standard module. 

\subsection{Induced of Lanlands' type}\label{SS:Langlands type}
Let $A_r\subset\GL_r$ be the diagonal torus. Given a tuple of $r$ nonzero complex numbers 
$\ul{\alpha}=(\alpha_1, \alpha_2\ldots, \alpha_r)$, there is a  unique unramified character $\chi_{\ul{\alpha}}$ of $A_r(F)$ 
given by 
\[
\chi_{\ul{\a}}({\rm diag}(a_1, a_2,\ldots, a_r))
=
\alpha_1^{-{\rm log}_q |a_1|_F}\alpha_2^{-{\rm log}_q|a_2|_F}\cdots \alpha_r^{-{\rm log}_q|a_r|_F}.
\]
By extending $\chi_{\ul{\a}}$ to a character of the Borel subgroup $B_r=A_r\ltimes Z_r\subset\GL_r$, we can form 
a normalized induced representation $\tau_{\ul{\alpha}}={\rm Ind}_{B_r(F)}^{\GL_r(F)}(\chi)$ of $\GL_r(F)$. This is 
an unramified representation of $\GL_r(F)$ whose $\GL_r(\frak{o})$-fixed subspace is one-dimensional. 
On the other hand, every unramified irreducible representation of $\GL_r(F)$ can be realized as a constituent of 
$\tau_{\ul{\alpha}}$ for some $\ul{\alpha}$. Note that $\tau^*_{\ul{\a}}=\tau_{\ul{\alpha}^*}$ with 
$\ul{\alpha}^*=(\alpha^{-1}_r, \alpha^{-1}_{r-1},\ldots,\alpha_1^{-1})$.\\ 

An unramified representation $\tau$ of $\GL_r(F)$ is called induced of $Langlands'$ $type$ if 
$\tau\cong\tau_{\ul{\alpha}}$ for some $\ul{\alpha}$ with
\begin{equation}\label{E:decreasing}
|\alpha_1|\le |\alpha_2|\le\cdots\le |\alpha_r|.
\end{equation}
These unramified representations may be reducible, but they have the following nice properties that allow us to work 
with: (i) The $\bbC$-linear space ${\rm Hom}_{Z_r(F)}(\tau, \b{\psi}_{Z_r})$ is one-dimensional.
(ii) The intertwining map $u\mapsto W_u$ from $u\in\cV_\tau$ to its associated Whittaker function is injective 
(\cite{JacquetShalika1983},\cite[Lemma 2]{Jacquet2012}). (iii) $\tau^*$ is again induced of Langlands' type. 
We denote by $J(\tau)$ the unique irreducible quotient of $\tau$, which is unramified (\cite[Corollary 1.2]{Matringe2013})
and has the Satake parameters $\alpha_1,\alpha_2, \ldots,\alpha_r$. 

\subsection{Satake isomorphisms}\label{SS:satake isom}
Let $\tau=\tau_{\ul{\a}}$ be an unramified representation of $\GL_r(F)$ (not necessarily induced of Langlands' type) and 
$v_{\tau}\in\cV_\tau^{\GL_r(\frak{o})}$ be nonzero. Let $\cS_r$ be the $\bbC$-algebra of symmetric polynomials in 
\[
(X_1, X_1^{-1}, X_2, X_2^{-1},\ldots, X_r, X_r^{-1}).
\]
It contains a subalgebra $\cS_r^0$ consisting of elements $f$ satisfying 
\[
f(X_1, X_2, \ldots, X^{-1}_i,\ldots, X^{-1}_j,\ldots,X_{r-1}, X_r)=f(X_1, X_2, \ldots, X_i,\ldots, X_j,\ldots,X_{r-1}, X_r)
\]
for all $1\le i< j\le r$. 
Let $\sS_r:\cH(\GL_r(F)//\GL_r(\frak{o}))\overset{\sim}{\longto}\cS_r$ be the Satake
isomorphism (\cite{Satake1963}). Then we have
\begin{equation}\label{E:satake for GL}
\int_{\GL_r(F)}\varphi(a)\tau(a)v_\tau da
=
\sS_r(\varphi)(\a_1,\a_2,\ldots,\a_r)v_\tau
\quad
({\rm vol}(\GL_r(\frak{o}), da)=1)
\end{equation}
for every $\varphi\in\cH(\GL_r(F)//\GL_r(\frak{o}))$. The algebra isomorphism $\sS_r$ is the 
composition of the following isomorphisms
\[
\cH(\GL_r(F)//\GL_r(\frak{o}))\overset{\varsigma_r}{\longto}\cH(A_r(F)//A_r(\frak{o}))^{W_{\GL_r}}
\cong\bbC[A_r(\bbC)]^{\frak{S}_r}=\cS_r
\] 
with
\[
\varsigma_r(\varphi)(d)
=
\delta^{\frac{1}{2}}_{B_r}(d)\int_{Z_r(F)}\varphi(dz)dz
\quad
({\rm vol}(Z_r(\frak{o}),dz)=1).
\]
Similarly, we have the algebra isomorphism
\begin{equation}\label{E:Satake for SO(2n)}
\sS^0_{r,m}: \cH(H_r(F)//R_{r,m})\overset{\varsigma_{r,m}^0}{\longto}\cH(T_r(F)//T_r(\frak{o}))^{W_{H_r}}
\cong\bbC[T_r(\bbC)]^{\frak{S}_r\ltimes(\bbZ/2\bbZ)^{r-1}}=\cS_r^0
\end{equation}
for each $m\ge 0$ (by \eqref{E:R*}) with
\begin{equation}\label{E:satake for H}
\varsigma^0_{r,m}(\varphi)(t)
=
\delta^{\frac{1}{2}}_{B_{H_r}}(t)\int_{V_r(F)}\varphi(tv)dv
\quad
({\rm vol}(V_r(\frak{o}),dv)=1).
\end{equation}
Here $W_{\GL_r}$ (resp. $W_{H_r}$) is the Weyl group of $\GL_r$ (resp. $H_r$) and $\frak{S}_r$ is the permutation 
group of order $r!$.\\

Following \cite[Section 5.3]{Tsai2013}, we define the map 
$\imath_{r,m}: \cH(H_r(F)//R_{r,m})\longto\cH(\GL_r(F)//\GL_r(\frak{o}))$ by 
\begin{equation}\label{E:middle satake}
\imath_{r,m}(\varphi)(a)
=
\delta^{\frac{1}{2}}_{Q_r}(m(a))\int_{N_r(F)}
\varphi(m(a)n)dn
\quad
({\rm vol}(N_r(\frak{o}),dn)=1). 
\end{equation}
Then one checks that
\begin{equation}\label{E:decompose of Satake isom}
\varsigma^0_{r,m}=\varsigma_r\circ\imath_{r,m}.
\end{equation}
Since the algebras $\cH(A_r(F)//A_r(\frak{o}))$ and $\cH(T_r(F)//T_r(\frak{o}))$
have a natural identification, \eqref{E:decompose of Satake isom} implies the following commutative diagram
\begin{equation}\label{E:diagram}
\begin{tikzcd}
&\cH(H_r(F)//R_{r,m})\quad\arrow{d}{\imath_{r,m}}\arrow{r}{\sS^0_{r,m}}&\quad\cS^0_r\arrow{d}{{\rm inclusion}}\\
&\cH(\GL_r(F)//\GL_r(\frak{o}))\quad\arrow{r}{\sS_r}&\quad\cS_r
\end{tikzcd}
\end{equation}
for every $m\ge 0$. These will be used to describe our results for oldforms.

\subsection{Spherical Whittaker functions}\label{SS:spherical Whittaker function}
Let $\tau=\tau_{\ul{\alpha}}$ be an unramified representation of $\GL_r(F)$ that is induced of Langlands' type. 
Fix nonzero elements $v_\tau\in\cV_\tau^{\GL_r(\frak{o})}$ and 
$\Lambda_{\tau,\b{\psi}}\in{\rm Hom}_{Z_r(F)}(\tau,\b{\psi}_{Z_r})$. Let
\[
W(a;\alpha_1,\alpha_2,\ldots,\alpha_r; \b{\psi})=\Lambda_{\tau,\b{\psi}}(\tau(a)v_{\tau})
\]
be a spherical Whittaker function on $\GL_r(F)$ with
\[
W(I_r;\alpha_1,\alpha_2,\ldots,\alpha_r; \b{\psi})=1.
\]
Explicit formulae for $W(-;\alpha_1,\alpha_2,\ldots,\alpha_r;\b{\psi})$ (when $\ker(\psi)=\frak{o}$) were obtained in 
\cite{Shintani1976}, \cite{CasselmanShalika1980} and can be described as follows. 
Let $\e_1,\e_2, \ldots, \e_r$ be the standard basis 
of $X_\bullet(A_r)={\rm Hom}(\bbG_m, A_r)$ such that $\e_j(y)\in A_r(F)$ ($1\le j\le r$) is the diagonal matrix whose 
$(j,j)$-entry is $y$, while all other diagonal entries are $1$. Put
\[
P^+_{\GL_r}
=
\stt{c_1\e_1+c_2\e_2+\cdots+c_r\e_r\mid c_1\ge c_2\ge\cdots\ge c_r}\subset X_\bullet(A_r).
\]
An element $a\in\GL_r(F)$ can be written as $a=z\varpi^\lambda k$ for some $z\in Z_r(F)$, 
$k\in\GL_r(\frak{o})$ and $\lambda\in X_\bullet(A_r)$ by the Iwasawa decomposition. Then it is not hard to show that 
$W(z\varpi^\lambda k;\alpha_1,\alpha_2, \ldots,\alpha_r;\b{\psi})=0$ when $\lambda\nin P^+_{\GL_r}$. 
On the other hand, when $\lambda\in P^+_{\GL_r}$, we have
\begin{equation}\label{E:Whittaker function for GL}
W(z\varpi^\lambda k;\alpha_1,\alpha_2, \ldots,\alpha_r;\b{\psi})
=
\b{\psi}_{Z_r}(z)
\cdot
\delta_{B_r}^{\frac{1}{2}}(\varpi^\lambda)
\cdot
\chi^{\GL_r}_\lambda(d_{\ul{\alpha}})
\end{equation}
where $\chi_\lambda^{\GL_r}$ is the character of the finite-dimensional representation of $\GL_r(\bbC)$ whose highest 
weight (with respect to $B_r$) is $\lambda$ and we denote 
$d_{\ul{\alpha}}={\rm diag}(\alpha_1,\alpha_2,\ldots,\alpha_r)\in A_r(\bbC)$.\\

In our discussion $|\alpha_1|\le |\alpha_2|\le\cdots\le|\alpha_r|$. However, the function 
\[
\ul{\alpha}\longmapsto W(a;\alpha_1,\alpha_2,\ldots,\alpha_r; \b{\psi})
\]
is polynomial and symmetric in $\alpha_j$'s for every fixed $a\in\GL_r(F)$. Consequently, we have a $\cS_r$-valued 
function 
\[
W(-:X_1, X_2,\ldots, X_r;\b{\psi}): \GL_r(F)\longto\cS_r
\]
on $\GL_r(F)$ such that for every $a\in\GL_r(F)$ and $r$-tuple of nonzero complex numbers 
$\ul{\a}=(\a_1,\a_2,\ldots,\a_r)$, the scalar $W(a;\a_1, \a_2,\ldots,\a_r;\b{\psi})$ is the value of the polynomial 
$W(a;X_1, X_2,\ldots, X_r; \b{\psi})$ at  $(\a_1,\a_2,\ldots, \a_r)$. This function was introduced by 
Jacquet, Piatetski-Shapiro and Shalika (\cite[Section 3]{JPSS1981}, \cite[Section 2]{Jacquet2012}) in order to prove 
the existence of the "essential vector" of an irreducible generic representation of $\GL_r(F)$. 
Note that for $a$ in a set compact modulo $Z_r(F)$, the polynomials remain in a finite-dimensional subspace of $\cS_r$. 
Moreover, we have the relation 
\[
W(a;\a_1,\a_2,\ldots,\a_r;\b{\psi})\nu_r(a)^s
=
W(a; q^{-s}\a_1, q^{-s}\a_2,\ldots, q^{-s}\a_r;\b{\psi}).
\]
This implies that if $\nu_r(a)=q^{-\ell}$, then the polynomial 
$W(a;X_1, X_2,\ldots, X_r;\b{\psi})$ is homogeneous of order $\ell$, i.e.
\begin{equation}\label{E:Whittaker homo}
Y^\ell W(a;X_1, X_2,\ldots,X_r;\b{\psi})
=
W(a; YX_1, YX_2,\ldots, YX_r;\b{\psi}).
\end{equation}

\section{A Key Construction}\label{S:key}
In this section, we extend the results in \cite[Chapter 5]{Tsai2013} from generic supercuspidal representations 
to generic representations as well as from $r=n$ to  $r\le n$. These results, which are stated in 
\propref{P:main prop}, are the core of our computations for the Rankin-Selberg integrals attached to (conjectural) 
newforms and oldforms.  To state and to prove \propref{P:main prop}; however, we need some preparations.

\subsection{Notation and conventions}\label{SS:convention}
Let $\ul{\alpha}=(\alpha_1,\alpha_2,\ldots,\alpha_r)$ be a $r$-tuple  of nonzero complex numbers and 
$\ul{\dot{\alpha}}$ be its rearrangement so that \eqref{E:decreasing} holds. 
Let $\tau=\tau_{\ul{\dot{\alpha}}}$, which is an unramified representation of $\GL_r(F)$ which is induced of 
Langlands' type (cf. \S\ref{SS:Langlands type}). Fix elements $v_\tau\in\cV_\tau^{\GL_r(\frak{o})}$ and 
$\Lambda_{\tau,\b{\psi}}\in{\rm Hom}_{Z_r(F)}(\tau,\b{\psi}_{Z_r})$ so that $\Lambda_{\tau,\b{\psi}}(v_\tau)=1$. 
By \eqref{E:R*} and \eqref{E:Iwasawa decomp}, the space $I(\tau,s)^{R_{r,m}}$ is one-dimensional and has the 
generator $\xi^m_{\tau,s}$ with $\xi^m_{\tau,s}(I_{2r})=v_\tau$. We have
\begin{equation}\label{E:f_xi, m}
f_{\xi^m_{\tau,s}}(m_r(a))
=
\Lambda_{\tau,\b{\psi}_{Z_r}}(\rho_{\tau,s}(m_r(a))\xi^m_{\tau,s})
=
W(a;\alpha_1,\alpha_2,\ldots,\alpha_r;\b{\psi})\nu_r(a)^{s+\frac{r}{2}-1}
\end{equation}
for $a\in\GL_r(F)$. Let $\pi$ be an irreducible generic representation of $G_n(F)$ and 
$\Lambda_{\pi,\psi}\in{\rm Hom}_{U_n(F)}(\pi,\psi_{U_n})$ be a nonzero element.

\subsubsection{Haar measures}\label{SSS:Haar}
In this and the next section, the Haar measures appeared in the Rankin-Selberg integrals are chosen as follows.
First, we take $d\b{x}$ to be the Haar measure on $\b{X}_{n,r}(F)$ with 
\[
{\rm vol}(\b{X}_{n,r}(\frak{o}),d\b{x})=1.
\]
On the other hand, the Haar measures $dt$ on $T_r(F)$ and $dk$ on $R_{r,m}$ (any $m\ge 0$) are chosen so that 
\[
{\rm vol}(T_r(\frak{o}),dt)={\rm vol}(R_{r,m},dk)=1.
\]
Then the quotient measure $dh$ on $V_r(F)\backslash H_r(F)$ is given by (cf. \eqref{E:Iwasawa decomp})
\[
\int_{V_r(F)\backslash H_r(F)}f(h)dh
=
\int_{T_r(F)}\int_{R_{r,m}}f(tk)\delta_{B_{H_r}}^{-1}(t)dtdk.
\]
Note that by \lmref{L:same vol}, $dh$ does not depend on the choice of $m$. 

\subsection{Some lemmas}
In this subsection, we collect some lemmas that will be used in the proof of \propref{P:main prop}.

\begin{lm}\label{L:general AL elt}
Let $u_{n,r,m}\in J_{n,m}$ be given by 
\begin{equation}\label{E:u_r,m}
u_{n,r,m}
=
\begin{pmatrix}
&&\varpi^{-m}I_r\\
&(-1)^rI_{2(n-r)+1}\\
\varpi^m I_r
\end{pmatrix}.
\end{equation}
Then $u_{n,r,m}$ normalizes both $K_{n,m}$ and $R_{n,m}$.
\end{lm}

\begin{proof}
Certainly, $u_{n,r,m}$ normalizes $K_{n,m}$ as it is contained in $J_{n,m}$. On the other hand, since $u_{n,r,m}$ also
normalizes $H_n(F)$ in $G_n(F)$, we find that 
\[
R_{n,m}=K_{n,m}\cap H_n(F)
=
u_{n,r,m}^{-1}K_{n,m}u_{n,r,m}\cap u_{n,r,m}^{-1}H_n(F)u_{n,r,m}
=
u_{n,r,m}^{-1}R_{n,m}u_{n,r,m}.
\]
This proves the lemma.
\end{proof}

\begin{lm}\label{L:supp for para vec}
Let $v\in\cV_\pi^{K_{n,m}}$ and $W_v$ be its associated Whittaker function. Then as a function of $t\in T_n(F)$, $W_v(t)$
is $T_n(\frak{o})$-invariant on the right and $W_v(\varpi^\la)=0$ if $\la\nin P^+_{G_n}$.
\end{lm}

\begin{proof}
This follows immediately from the fact that $T_n(\frak{o})$ and $U_n(\frak{o})$ are both contained in $K_{n,m}$.
\end{proof}

Recall that $\gamma(s,\tau,\bigwedge^2,\psi)$ is Shahidi's local coefficient introduced in \S\ref{SSS:intertwining map}.

\begin{lm}\label{L:LS=Gal}
We have $\gamma(s,\tau,\bigwedge^2,\psi)=\gamma(s,\phi_{J(\tau)},\bigwedge^2,\psi)$.
\end{lm}

\begin{proof}
This is a consequence of the multiplicativity of $\gamma(s,\tau,\bigwedge^2,\psi)$ (which holds even when $\tau$ is 
reducible under our assumptions on $\tau$) and a result of Shahidi
(\cite[Theorem 3.5]{Shahidi1990}) together with a recent result of Cogdell-Shahidi-Tsai 
(\cite[Theorem 1.1]{CogdellShahidiTsai2017}). 
\end{proof}

The next lemma computes the action of the intertwining map on $\xi_{\tau,s}^m$, which is crucial in the proof of 
\propref{P:main prop}.
Note that $\delta_r^{-1}w_{r,m}\in H_r(F)$, where $\delta_r$ and $w_{r,m}$ is given by \eqref{E:delta} and 
\eqref{E:t and w} respectively.

\begin{lm}\label{L:GK method for m}
We have 
\[
\rho(\delta_r^{-1} w_{r,m})M(\tau, s)\xi^m_{\tau,s}(\delta_r^{-1}h\delta_r)
=
\omega_{\tau_s}(\varpi)^m\cdot \frac{L(2s-1,\phi_{J(\tau)},\bigwedge^2)}{L(2s, \phi_{J(\tau)}, \bigwedge^2)}\cdot
\xi^m_{\tau^*,1-s}(h)
\]
for $m\ge 0$, where $\omega_{\tau_s}$ is the central character of $\tau_s$.
\end{lm}

\begin{proof}
Since $\rho(\delta_r^{-1} w_{r,m})M(\tau, s)\xi^m_{\tau,s}(\delta_r^{-1}h\delta_r)\in I_r(\tau^*, 1-s)^{R_{r,m}}$ by the 
second conjugation in \eqref{E:R conj}, it suffices to show 
\[
\rho(\delta_r^{-1} w_{r,m})M(\tau, s)\xi^m_{\tau,s}(I_{2r})
=
\omega_{\tau_s}(\varpi)^m \frac{L(2s-1,\phi_{J(\tau)},\bigwedge^2)}{L(2s, \phi_{J(\tau)}, \bigwedge^2)}.
\]
For this, we first note that the integral \eqref{E:intertwining map} can be written as 
\[
M(\tau,s)\xi_s(h)=\int_{N_r(F)}\xi_s(w_{r,0}n\delta_r h)dn.
\]
If $m=0$, then $\delta_r^{-1}w_{r,0}\in R_{r,0}=H_r(\frak{o})$ and assertion can be read off
from the formula in \cite[Section 4]{Arthur1981}, i.e. we have 
\[
M(\tau,s)\xi^0_{\tau,s}(I_{2r})
=
\int_{N_r(F)}\xi^0_{\tau,s}(w_{r,0}n\delta_r)dn
=
\frac{L(2s-1,\phi_{J(\tau)},\bigwedge^2)}{L(2s, \phi_{J(\tau)}, \bigwedge^2)}.
\]
Note that the Haar measure $dn$ on $N_r(F)$ (cf. \S\ref{SSS:intertwining map}) is the same as the one used in 
\cite[Section 4]{Arthur1981}. Now suppose that $m>0$. On one hand, by the first conjugation in \eqref{E:R conj}, we have 
$(\xi^0_{\tau,s})^{t_{r,m}}=\xi^m_{\tau,s}$, where $(\xi^0_{\tau,s})^{t_{r,m}}$ is the function on $H_r(F)$ defined by 
$(\xi^0_{\tau,s})^{t_{r,m}}(h)=\xi^0_{\tau,s}(t_{r,m}^{-1}ht_{r,m})$. On the other hand, since 
$t_{r,m}^{-1}w_{r,0}t_{r,m}=\varpi^{m\mu_r}w_{r,0}$, we find that 
\begin{align*}
\rho(\delta^{-1} w_{r,m})M(\tau, s)\xi^m_{\tau,s}(I_{2r})
&=
M(\tau, s)(\xi^0_{\tau,s})^{t_{r,m}}(\delta_r^{-1}w_{r,m})\\
&=
\int_{N_r(F)}\xi^0_{\tau,s}(t^{-1}_{r,m}w_{r,0}n w_{r,m}t_{r,m})dn\\
&=
\int_{N_r(F)}\xi^0_{\tau,s}(\varpi^{m\lambda_r}w_{r,0} (t^{-1}_{r,m}n\,t_{r,m})w_{r,0})dn\\
&=
\omega_{\tau_s}(\varpi)^m\int_{N_r(F)}\xi^0_{\tau,s}(w_{r,0}n\,\delta_r(\delta_r^{-1}w_{r,0}))dn\\
&=
\omega_{\tau_s}(\varpi)^m \frac{L(2s-1,\phi_{J(\tau)},\bigwedge^2)}{L(2s, \phi_{J(\tau)}, \bigwedge^2)}.
\end{align*}
This concludes the proof.
\end{proof}

The following lemma helps us to simplify our computations.

\begin{lm}\label{L:supp of W_v}
Let $v\in\cV_\pi^{K_{n,m}}$, $a\in\GL_r(F)$ and $\b{x}\in \b{X}_{n,r}(F)\setminus \b{X}_{n,r}(\frak{o})$ with $r<n$.
Then $W_v(m_r(a)\b{x})=0$.
\end{lm}

\begin{proof}
Let us denote $E_{ij}\in{\rm Mat}_{(2n+1)\x(2n+1)}(F)$ to be the matrix with $1$ in the $(i,j)$ entry and $0$ in all 
other entries. For $x=(x_{ij})\in{\rm Mat}_{(n-r)\x r}(F)$, we set
\[
\b{x}
=
\begin{pmatrix}I_r&&&&\\x&I_{n-r}&&\\&&1\\&&&I_{n-r}&\\&&&x'&I_r\end{pmatrix}
\in
\b{X}_{n,r}(F).
\]
Suppose that $\b{x}\notin\b{X}_{n,r}(\frak{o})$. Then there exist $1 \le\ell\le n-r$ and $1\le k\le r$ so that 
$x_{\ell k}\notin\frak{o}$, but $x_{ij}\in\frak{o}$ if $\ell<i\le n-r$ or $i=\ell$ and $m< j\le r$. Define 
$x_0=(x^0_{ij})\in{\rm Mat}_{(n-r)\x r}(F)$ by $x^0_{ij}=x_{ij}$ if $i< \ell$ or 
$i=\ell$ and $1\le j\le m$, while $x^0_{ij}=0$ otherwise. Then we have 
\[
W_v(m_r(a)\b{x})=W_v(m_r(a)\b{x}_0)
\]
since $\b{x}_0^{-1}\b{x}\in K_{n,m}$ by our assumption on $\b{x}$.
To prove the lemma, we first assume that $a=t\in A_r(F)$.  In this case, the idea is to find 
$u\in U_n(\frak{o})\subset K_{n,m}$ such that 
\begin{equation}\label{E:u}
W_v(m_r(t)\b{x}_0)
=
W_v(m_r(t)\b{x}_0u)
=
\psi(yx_{\ell m})W_v(m_r(t)\b{x}_0)
\end{equation}
where $y\in\frak{o}$ appears in the entries of $u$. Since $\ker(\psi)=\frak{o}$, $x_{\ell m}\notin\frak{o}$ and we can 
let $y\in\frak{o}$ be arbitrary, the proof for $a\in A_r(F)$ will follow. To define $u$, let $y\in\frak{o}$ and put
\[
u=I_{2n+1}+yE_{m, \ell+1}-yE_{2n+1-\ell, 2n+2-m}
\]
if $\ell<n-r$, whereas
\[
u=I_{2n+1}-2yE_{m,n+1}+yE_{n+1,2n+2-m}-y^2E_{m,2n+2-m}
\]
if $\ell=n-r$. One then checks that $u\in U_n(\frak{o})$ and \eqref{E:u} holds.
For arbitrary $a\in\GL_r(F)$, we can write $a=ztk$ for some $z\in Z_r(F)$, $t\in A_r(F)$ and $k\in\GL_r(\frak{o})$. 
Then 
\[
W_v(m_r(a)\b{x})
=
\psi_{U_n}(m_r(z))\cdot W_v(m_r(t)\b{x}')
=
0
\]
since $m_r(k)\in K_{n,m}$ and $\b{x}':=m_r(k)\b{x}m_r(k)^{-1}\notin\b{X}_{n,r}(\frak{o})$.
\end{proof}

\begin{lm}\label{L:exp of RS int}
Let $v\in\cV_\pi^K$ with $K=K_{n,m}$ if $r<n$ and $K=R_{r,m}$ if $r=n$. Then we have 
\begin{align*}
\Psi_{n,r}(v\ot\xi^m_{\tau,s})
&=
\sum_{\ell\in\bbZ}\int_{a\in Z_r(F)\backslash\GL_r(F),\,\nu_r(a)=q^{-\ell}}
W_v(m_r(a))W(a;\a_1,\a_2,\ldots,\a_r;\b{\psi})\nu_r(a)^{s-n+\frac{r}{2}}da
\end{align*}
for $\Re(s)\gg 0$ with the $\ell$-th summand vanishes for all $\ell\gg 0$.
\end{lm}

\begin{proof}
In the following computations, we assume implicitly that $s$ is in the domain of convergence. First note that we have the 
identification
\[
V_r(F)\backslash Q_r(F)\cong Z_r(F)\backslash\GL_r(F).
\]
Then by the Iwasawa decomposition $H_r(F)=Q_r(F)R_{r,m}$ together with \eqref{E:RS integral in general} and 
\eqref{E:f_xi, m}, we find that 
\begin{align*}
\Psi_{n,n}(v\ot\xi^m_{\tau,s})
&=
\int_{ Z_n(F)\backslash\GL_n(F)}
W_v(m_n(a))W(a;\a_1,\a_2,\ldots,\a_n;\b{\psi})\nu_n(a)^{s-\frac{n}{2}}da\\
&=
\sum_{\ell\in\bbZ}\int_{a\in Z_n(F)\backslash\GL_n(F),\,\nu_n(a)=q^{-\ell}}
W_v(m_n(a))W(a;\a_1,\a_2,\ldots,\a_n;\b{\psi})\nu_n(a)^{s-\frac{n}{2}}da
\end{align*}
if $r=n$. On the other hand, by \lmref{L:supp of W_v} and the fact that both $R_{r,m}$ and 
$\b{X}_{n,r}(\frak{o})$ are contained in $K_{n,m}$, we get that
\begin{align*}
\Psi_{n,r}(v\ot\xi^m_{\tau,s})
&=
\int_{Z_r(F)\backslash\GL_r(F)}\int_{\b{X}_{n,r}(F)}
W_v(\b{x}m_r(a))W(a;\a_1,\a_2,\ldots,\a_r;\b{\psi})\nu_r(a)^{s-\frac{r}{2}}d\b{x}da\\
&=
\int_{Z_r(F)\backslash\GL_r(F)}\int_{\b{X}_{n,r}(F)}
W_v(m_r(a)\b{x})W(a;\a_1,\a_2,\ldots,\a_r;\b{\psi})\nu_r(a)^{s-n+\frac{r}{2}}d\b{x}da\\
&=
\int_{Z_r(F)\backslash\GL_r(F)}
W_v(m_r(a))W(a;\a_1,\a_2,\ldots,\a_r;\b{\psi})\nu_r(a)^{s-n+\frac{r}{2}}da\\
&=
\sum_{\ell\in\bbZ}\int_{a\in Z_r(F)\backslash\GL_r(F),\,\nu_r(a)=q^{-\ell}}
W_v(m_r(a))W(a;\a_1,\a_2,\ldots,\a_r;\b{\psi})\nu_r(a)^{s-n+\frac{r}{2}}da
\end{align*}
if $r<n$. This proves the first assertion. To prove the $\ell$-th summand vanishes for all $\ell\gg 0$, we can use the 
the Iwasawa decomposition of $\GL_r$ to derive (recall $R_{r,m}\cap M_r(F)=M_r(\frak{o})$)
\begin{align}\label{E:exp of RS int}
\begin{split}
&\int_{a\in Z_r(F)\backslash\GL_r(F),\,\nu_r(a)=q^{-\ell}}
W_v(m_r(a))W(a;\a_1,\a_2,\ldots,\a_r;\b{\psi})\nu_r(a)^{s-n+\frac{r}{2}}da\\
&=
\sum_{\la\in P^+_{\GL_r},\,\,{\rm tr}(\la)=\ell}
W_v(m_r(\varpi^\la))W(\varpi^\la;\a_1,\a_2,\ldots,\a_r;\b{\psi})\nu_r(\varpi^\la)^{\frac{1}{2}-n+\frac{r}{2}}.
\end{split}
\end{align}
Now if $\la=\la_1\e_1+\la_2\e_2+\cdots+\la_r\e_r\in P^+_{\GL_r}$, then there exists $N\in\bbZ$ depending only on 
$v$, such that $W_v(\varpi^\la)=0$ whenever $\la_j<N$ for some $1\le j\le r$. Actually, we have $N=0$ if $r<n$ by
\lmref{L:supp for para vec}. From this one sees that the summation in \eqref{E:exp of RS int} is a finite sum and each of 
its term vanishes when $\ell\gg 0$. This completes the proof.
\end{proof}

\subsection{The key proposition}
Now we are in the position to state and prove \propref{P:main prop}. As mentioned in the beginning of this section, this 
proposition is the key to the proofs of our main results. Its proof, on the other hand, can be viewed as a generalization of
the one given by Jacquet, Piatetski-Shapiro and Shalika (\cite{JPSS1981}, \cite{JPSS1983}) to our case. We note that 
Kaplan also used the same technique in \cite{Kaplan2013} to establish the relations between the $L$-factors for generic representations of $\SO_{2n}\x\GL_r$ defined by the Rankin-Selberg integrals and Langlands-Shahidi's method.

\begin{prop}\label{P:main prop}
Let $K=K_{n,m}$ if $r<n$ and $K=R_{n,m}$ if $r=n$. Then there exist linear maps
\[
\Xi^m_{n,r}: \cV_{\pi}^{K}\longto\cS_r
\]
sending $v$ to $\Xi^m_{n,r}(v;X_1, X_2,\ldots,X_r)$ such that the followings are satisfied:
\begin{itemize}
\item[(1)]
We have
\[
\Xi^m_{n,r}(v;q^{-s+1/2}\alpha_1,q^{-s+1/2}\alpha_2, \ldots,q^{-s+1/2}\alpha_r)
=
\frac{L(2s,\phi_{J(\tau)},\bigwedge^2)\Psi_{n,r}(v\ot \xi^m_{\tau,s})}{L(s,\phi_\pi\ot\phi_{J(\tau)})}
\]
for every $s\in\bbC$.
\item[(2)]
We have $\pi(u_{n,r,m})v\in\cV_\pi^K$ and the functional equation
\[
\Xi^m_{n,r}(\pi(u_{n,r,m})v; X_1^{-1},X_2^{-1}, \ldots, X_r^{-1})=\e_\pi^r(X_1X_2\cdots X_r)^{(a_\pi-m)}
\Xi^m_{n,r}(v;X_1, X_2, \ldots,X_r)
\]
holds, where $u_{n,r,m}\in G_n(F)$ is given by \eqref{E:u_r,m}.
\item[(3)]
The kernel of $\Xi^m_{n,r}$ is given by 
\[
ker(\Xi^m_{n,r})=\stt{v\in \cV_\pi^K\mid \text{$W_v(t)=0$ for every $t\in T_r$}}.
\]
\item[(4)]
The relation 
\[
\Xi^m_{n,r}(v;X_1,X_2, \ldots,X_{r-1},0)=\Xi^m_{n,r-1}(v;X_1, X_2\ldots,X_{r-1})
\]
holds for $v\in\cV_\pi^{K_{n,m}}$ and $2\leq r\leq n$.
\item[(5)]
We have
\[
\Xi^m_{n,n}(\varphi\star v;X_1, X_2,\ldots, X_n)
=\sS^0_{n,m}(\varphi)\cdot \Xi^m_{n,n}(v; X_1, X_2,\ldots, X_r)
\]
for $\varphi\in\cH(H_n(F)//R_{n,m})$, where $\sS^0_{n,m}$ is the algebra isomorphism given by 
\eqref{E:Satake for SO(2n)}.
\end{itemize}
\end{prop}

We divide the proof into five parts, one for each item.

\subsubsection{Proof of $(1)$}
Let $v\in\cV_\pi^K$. Motivated by \eqref{E:exp of RS int}, we consider the following integral 
\begin{align}\label{E:Psi_r,m,l}
\begin{split}
\Psi^m_{n,r,\ell}(v; X_1, X_2,\ldots, X_r)
&=
\int_{a\in Z_r(F)\backslash\GL_r(F),\,\nu_r(a)=q^{-\ell}}
W_v(m_r(a))W(a;X_1,X_2,\ldots,X_r;\b{\psi})\nu_r(a)^{\frac{r+1}{2}-n}da\\
&=
\sum_{\la\in P^+_{\GL_r},\,\,{\rm tr}(\la)=\ell}
W_v(m_r(\varpi^\la))W(\varpi^\la;X_1,X_2,\ldots,X_r;\b{\psi})\nu_r(\varpi^\la)^{\frac{r+1}{2}-n}.
\end{split}
\end{align}
We know this is a finite sum and hence gives rise to a homogeneous polynomial of 
degree $\ell$ in $\cS_r$ by \eqref{E:Whittaker homo}. Furthermore, there exists an integer $N_{v}$ depending only on 
$v$ such that $\Psi^m_{n,r,\ell}(v; X_1, X_2,\ldots, X_r)=0$ for all $\ell<N_v$. We then define the following formal 
Laurent series 
\begin{equation}\label{E:Psi_r,m}
\Psi^m_{n,r}(v;X_1,X_2,\ldots, X_r;Y)
=
\sum_{\ell\in\bbZ}\Psi^m_{n,r,\ell}(v;X_1,X_2,\ldots, X_r)Y^\ell
=
\sum_{\ell\ge N_{v}}\Psi^m_{n,r,\ell}(v;X_1,X_2,\ldots, X_r)Y^\ell
\end{equation}
with coefficient in $\cS_r$. Now \lmref{L:exp of RS int} implies 
\begin{equation}\label{E:Psi eva}
\Psi^m_{n,r}(v;\a_1,\a_2,\ldots,\a_r; q^{-s+\frac{1}{2}})
=
\Psi_{n,r}(v\ot\xi^m_{\tau,s})
\end{equation}
for $\Re(s)\gg 0$. On the other hand, let $P_{\phi_\pi}(Y)\in\bbC[Y]$ such that $L(s,\phi_\pi)=P_{\phi_\pi}(q^{-s})^{-1}$ 
and put
\[
P_{\phi_\pi}(X_1, X_2,\ldots, X_r; Y)
=
\prod_{j=1}^r P_{\phi_\pi}(q^{-\frac{1}{2}}X_jY)
=
\sum_{\ell\ge 0} a_\ell(X_1, X_2,\ldots, X_r)Y^\ell.
\]
Then $a_\ell(X_1, X_2, \ldots, X_r)$ is a homogeneous polynomial of degree $\ell$ in $\cS_r$
with $a_0(X_1, X_2, \ldots, X_r)=1$. We have 
\[
L(s,\phi_\pi\ot\phi_{J(\tau)})
=
P_{\phi_\pi}(\a_1,\a_2,\ldots, \a_r;q^{-s+\frac{1}{2}})^{-1}
\quad\text{and}\quad
L(1-s,\phi_\pi\ot\phi_{J(\tau^*)})
=
P_{\phi_\pi}(\a^{-1}_1,\a^{-1}_2,\ldots, \a^{-1}_r;q^{s-\frac{1}{2}})^{-1}.
\]
We also put
\[
P_{\bigwedge^2}(X_1, X_2,\ldots,X_r; Y)
=
\prod_{1\le i<j\le r}(1-q^{-1}X_iX_jY^2).
\]
Then 
\[
L(2s,\phi_{J(\tau)},\Wedge{}^2)
=
P_{\bigwedge^2}(\a_1, \a_2,\ldots,\a_r; q^{-s+\frac{1}{2}})^{-1}
\quad\text{and}\quad
L(2-2s,\phi_{J(\tau^*)},\Wedge{}^2)
=
P_{\bigwedge^2}(\a^{-1}_1, \a^{-1}_2,\ldots,\a^{-1}_r; q^{s-\frac{1}{2}})^{-1}.
\]
By the geometric series expansions, one can write
\[
P_{\bigwedge^2}(X_1, X_2,\ldots,X_r; Y)^{-1}
=
\sum_{\ell\ge 0}
b_\ell(X_1,X_2,\ldots, X_r)Y^\ell.
\]
Again, $b_\ell(X_1, X_2, \ldots, X_r)$ is a 
homogeneous polynomial of degree $\ell$ in $\cS_r$ with $b_0(X_1, X_2, \ldots, X_r)=1$. But certainly this is not a finite 
sum.\\

At this point, let us put 
\begin{equation}\label{E:Xi_r,m}
\Xi^m_{n,r}(v;X_1,X_2,\ldots,X_r;Y)
=
\frac{P_{\phi_\pi}(X_1, X_2,\ldots, X_r; Y)\Psi^m_{n,r}(v;X_1,X_2,\ldots,X_r;Y)}
{P_{\bigwedge^2}(X_1, X_2,\ldots,X_r; Y)}
\end{equation}
which is again a formal Laurent series with coefficients in $\cS_r$. It can be written as
\begin{equation}\label{E:exp Xi_r,m,l}
\Xi^m_{n,r}(v;X_1,X_2,\ldots,X_r;Y)
=
\sum_{\ell\ge N_{v}}
\Xi^m_{n,r,\ell}(v;X_1,X_2\ldots,X_r)Y^\ell
\end{equation}
with
\begin{equation}\label{E:Xi_r,m,l}
\Xi^m_{n,r,\ell}(v;X_1,X_2\ldots,X_r)
=
\sum_{\ell_1+\ell_2+\ell_3=\ell}
\Psi^m_{n,r,\ell_1}(v;X_1,X_2\ldots,X_r)
a_{\ell_2}(X_1, X_2,\ldots, X_r)
b_{\ell_3}(X_1,X_2,\ldots, X_r)
\end{equation}
which is a finite sum. Clearly, $\Xi^m_{n,r,\ell}(v;X_1,X_2\ldots,X_r)$ is a homogeneous polynomial of degree $\ell$ in 
$\cS_r$. Also, it follows from \eqref{E:Psi eva} that
\begin{equation}\label{E:formal RS int eva}
\Xi^m_{n,r}(v;\a_1,\a_2,\ldots,\a_r; q^{-s+\frac{1}{2}})
=
\frac{L(2s,\phi_{J(\tau)},\bigwedge^2)\Psi_{n,r}(v\ot\xi^m_{\tau,s})}{L(s,\phi_\pi\ot\phi_{J(\tau)})}
\end{equation}
for $\Re(s)\gg 0$.
\\

A priori, the sum in \eqref{E:exp Xi_r,m,l} is infinite; however, we will show that it is indeed a finite sum by using the 
functional equation \eqref{E:FE}. 
To be more precise, let $w=\delta_r^{-1}w_{r,m}\in H_r(F)$ and $\h{w}\in G_n(F)$ be its image 
under the embedding \eqref{E:embedding}.  Then one checks that 
\[
\delta_{n,r}\h{w}
=
u_{n,r,m}
\]
where $\delta_{n,r}$ and $u_{n,r,m}$ are elements in $G_n(F)$ given by \eqref{E:delta_n,r} and 
\eqref{E:u_r,m}, respectively.
Now by \eqref{E:dual RS integral in general}, \lmref{L:LS=Gal} and \lmref{L:GK method for m}, we have
\begin{align*}
\t{\Psi}_{n,r}(v\ot\xi^m_{\tau,s})
=
\t{\Psi}_{n,r}(\pi(\h{w})v\ot\rho(w)\xi^m_{\tau,s})
=
\omega_{\tau_s}(\varpi)^m\frac{L(2-2s,\phi_{J(\tau^*)},\bigwedge^2)}{L(2s,\phi_{J(\tau)},\bigwedge^2)}
\Psi_{n,r}(\pi(u_{n,r,m})v\ot\xi^m_{\tau^*,1-s}).
\end{align*}
On the other hand, we also have
\[
\gamma(s,\pi\x\tau,\psi)
=
\epsilon(s,\phi_\pi\ot\phi_{J(\tau)},\psi)\frac{L(1-s,\phi_\pi\ot\phi_{J(\tau^*)})}{L(s,\phi_\pi\ot\phi_{J(\tau)})}
\]
by \thmref{T:RS=Gal gamma}. Together, we find that  \eqref{E:FE} can be written as 
\begin{equation}\label{E:FE sp}
\frac{L(2-2s,\phi_{J(\tau^*)},\bigwedge^2)\Psi_{n,r}(\pi(u_{n,r,m})v\ot\xi^m_{\tau^*,1-s})}{L(1-s,\phi_\pi\ot\phi_{J(\tau^*)})}
=
\omega_{\tau_s}(\varpi)^{-m}\epsilon(s,\phi_\pi\ot\phi_{J(\tau)},\psi)
\frac{L(2s,\phi_{J(\tau)},\bigwedge^2)\Psi_{n,r}(v\ot\xi^m_{\tau,s})}{L(s,\phi_\pi\ot\phi_{J(\tau)})}.
\end{equation}
Now because $u_{n,r,m}$ normalizes $K$ by \lmref{L:general AL elt}, we have $\pi(u_{n,r,m})v\in\cV_\pi^K$, and
hence \lmref{L:exp of RS int} gives
\begin{align*}
\Psi_{n,r}(\pi(&u_{n,r,m})v\ot\xi^m_{\tau^*,1-s})\\
&=
\sum_{\ell\in\bbZ}\int_{a\in Z_r(F)\backslash\GL_r(F),\,\nu_r(a)=q^{-\ell}}
W_{\pi(u_{n,r,m})v}(m_r(a))W(a;\a^{-1}_1,\a^{-1}_2,\ldots,\a^{-1}_r;\b{\psi})\nu_r(a)^{1-s-n+\frac{r}{2}}da
\end{align*}
for $\Re(s)\ll 0$. It follows that the formal Laurent series 
\[
\Xi^m_{n,r}(\pi(u_{n,r,m})v;X_1, X_2,\ldots, X_r;Y)
=
\sum_{\ell\ge -N_{\pi(u_{n,r,m})v}}
\Xi^m_{n,r,\ell}(\pi(u_{n,r,m})v;X_1,X_2\ldots,X_r)Y^\ell
\]
satisfies 
\begin{equation}\label{E:dual formal RS int eva}
\Xi^m_{n,r}(\pi(u_{n,r,m})v;\a_1^{-1},\a_2^{-1},\ldots,\a_r^{-1}; q^{s-\frac{1}{2}})
=
\frac{L(2-2s,\phi_{J(\tau^*)},\bigwedge^2)\Psi_{n,r}(\pi(u_{n,r,m})v\ot\xi^m_{\tau^*,1-s})}{L(1-s,\phi_\pi\ot\phi_{J(\tau^*)})}
\end{equation}
for $\Re(s)\ll 0$. Let us upgrade the functional equation \eqref{E:FE sp} to the one between formal Laurent series. 
For this, we put
\[
\epsilon_{\phi_\pi,m}(X_1,X_2,\ldots,X_r;Y)
=
\e^r_\pi(X_1X_2\cdots X_r)^{a_\pi-m}Y^{(a_\pi-m)r}.
\]
This is a unit in $\cS_r[Y]$ and we have
\[
\epsilon_{\phi_\pi,m}(\a_1,\a_2,\ldots,\a_r;q^{-s+\frac{1}{2}})
=
\omega_{\tau_s}(\varpi)^{-m}\epsilon(s,\phi_\pi\ot\phi_{J(\tau)},\psi)
\]
for $s\in\bbC$. Applying \eqref{E:formal RS int eva}, \eqref{E:dual formal RS int eva} and 
\eqref{E:FE sp}, we get that
\begin{equation}\label{E:FE for formal RS int}
\Xi_{r,m}(\pi(u_{n,r,m})v;X_1^{-1},X_2^{-1},\ldots, X_r^{-1}; Y^{-1})
=
\epsilon_{\phi_\pi,m}(X_1, X_2,\ldots, X_r;Y)
\Xi_{r,m}(v;X_1,X_2,\ldots,X_r;Y).
\end{equation}
This implies that \eqref{E:exp Xi_r,m,l} is a finite sum and hence we can define
\[
\Xi^m_{n,r}(v;X_1,X_2,\ldots,X_r)
=
\Xi^m_{n,r}(v;X_1,X_2,\ldots,X_r;1)\in\cS_r.
\]
Observe that
\begin{equation}\label{E:Xi relation}
\Xi^m_{n,r}(v;YX_1, YX_2,\ldots, YX_r)=\Xi^m_{n,r}(v;X_1,X_2,\ldots,X_r;Y)
\end{equation}
by the fact that $\Xi^m_{n,r,\ell}(v;X_1,X_2,\ldots,X_r)$ is a homogeneous polynomial of degree $\ell$. 
Then \eqref{E:formal RS int eva} implies
\begin{align*}
\Xi^m_{n,r}(v;q^{-s+\frac{1}{2}}\a_1, q^{-s+\frac{1}{2}}\a_2,\ldots, q^{-s+\frac{1}{2}}\a_r)
=
\frac{L(2s,\phi_{J(\tau)},\bigwedge^2)\Psi_{n,r}(v\ot\xi^m_{\tau,s})}{L(s,\phi_\pi\ot\phi_{J(\tau)})}
\end{align*}
for $s\in\bbC$. This defines the linear map $\Xi_{n,r}^m$ and completes the proof of $(1)$.\qed

\subsubsection{Proof of $(2)$}
This follows immediately by putting $Y=1$ in \eqref{E:FE for formal RS int}.\qed

\subsubsection{Proof of $(3)$}
We first note that $m_r(A_r(F))=T_r(F)$ (cf. \eqref{E:levi of Q}). Also, since $m_r(A_r(\frak{o}))\subset R_{r,m}$, 
it suffices to show
\[
\Xi^m_{n,r}(v;X_1,X_2,\ldots,X_r)=0
\quad\text{if and only if}\quad
W_v(m_r(\varpi^\la))=0\quad\text{for all $\la\in X_\bullet(A_r)$}.
\]
For a given $\la\in X_\bullet(A_r)$, let us denote
\[
\lambda=\la_1\e_1+\lambda_2\e_2+\cdots+\lambda_r\e_r.
\] 
Then we know that $W_v(m_r(\varpi^\la))=0$ if $\la_j< N$ for some $1\le j\le r$, where $N$ is an integer depending 
on $v$. Since $m_r(\GL_r(\frak{o}))\subset K$, a standard arguments show that  $W_v(m_r(\varpi^\lambda))=0$ if 
$\lambda\nin P^+_{\GL_r}$. Together, we obtain
\begin{equation}\label{E:rought support}
W_v(m_r(\varpi^\lambda))=0\quad\text{if $\lambda\nin P^+_{\GL_r}$ or if $\la\in P^+_{\GL_r}$ with $\la_r<N$}.
\end{equation}
Hence, we are reduced to prove
\begin{equation}\label{E:Xi=0}
\Xi^m_{n,r}(v;X_1,X_2,\ldots,X_r)=0
\quad\text{if and only if}\quad
W_v(m_r(\varpi^\la))=0\quad\text{for all $\la\in P^+_{\GL_r}$ with $\la_r\ge N$}.
\end{equation}
By \eqref{E:Xi relation}, \eqref{E:Xi_r,m} and \eqref{E:Psi_r,m}, we find that 
\begin{equation}\label{E:Xi=0 equiv}
\Xi^m_{n,r}(v;X_1,X_2,\ldots,X_r)=0
\quad\text{if and only if}\quad
\Psi^m_{n,r,\ell}(v;X_1,X_2,\ldots,X_r)=0\quad\text{for all $\ell\in\bbZ$}.
\end{equation}
This is equivalent to 
\[
\Psi^m_{n,r,\ell}(v;\a_1,\a_2,\ldots,\a_r)=0
\]
for all $\ell\in\bbZ$ and $r$-tuple $\ul{\a}=(\a_1,\a_2,\ldots,\a_r)$ of nonzero complex numbers. 
Now by \eqref{E:Whittaker function for GL}, \eqref{E:Psi_r,m,l} and \eqref{E:rought support}, we find that  
\begin{equation}\label{E:exp Psi_r,m,l}
\Psi^m_{n,r,\ell}(v;\a_1,\a_2,\ldots,\a_r)
=
q^{-\ell\left(\frac{r+1}{2}-n\right)}
\sum_{\lambda\in \Upsilon_{\ell}}
W_v(m_r(\varpi^\lambda))\delta^{\frac{1}{2}}_{B_r}(\varpi^\la)\chi^{\GL_r}_\lambda(d_{\ul{\a}})
\end{equation}
where
\[
\Upsilon_\ell
=
\stt{\lambda\in P^+_{\GL_r}\mid {\rm tr}(\lambda)=\ell\,\,\text{and}\,\,\la_r\ge N}
\]
with ${\rm tr}:X_\bullet(A_r)\to\bbZ$ the trace map (with respect to $\e_1, \e_2,\ldots, \e_r$) similar to \eqref{E:trace}. 
Evidently, $\Upsilon_\ell$ is a finite set for all $\ell\in\bbZ$, and is in fact the empty set if $\ell\ll 0$.
Since $\delta^{\frac{1}{2}}_{B_r}(\varpi^\la)\ne 0$ and $\ul{\a}$ can be arbitrary, we conclude from the linear 
independence of characters and \eqref{E:exp Psi_r,m,l} that 
\[
\Psi^m_{n,r,\ell}(v;X_1,X_2,\ldots,X_r)=0
\quad\text{if and only if}\quad
W_v(m_r(\varpi^\la))=0\quad\text{for all $\la\in\Upsilon_\ell$}.
\]
Now \eqref{E:Xi=0} follows from this together with \eqref{E:Xi=0 equiv}.\qed

\subsubsection{Proof of $(4)$}
It is clear from the definitions that 
\[
P_{\phi_\pi}(X_1,X_2,\ldots,X_{r-1},0; Y)
=
P_{\phi_\pi}(X_1,X_2,\ldots,X_{r-1}; Y)
\]
and
\[
P_{\bigwedge^2}(X_1,X_2,\ldots,X_{r-1},0; Y)
=
P_{\bigwedge^2}(X_1,X_2,\ldots,X_{r-1}; Y).
\]
Thus the proof will follow if we could show
\begin{equation}\label{E:r to r-1}
\Psi^m_{n,r}(v;X_1,X_2,\ldots, X_{r-1},0;Y)
=
\Psi^m_{n,r-1}(v;X_1,X_2,\ldots, X_{r-1};Y)
\end{equation}
for $v\in\cV_\pi^{K_{n,m}}$. For this, we follow the idea of Ginzburg in \cite[Theorem B]{Ginzburg1990} to "remove" the
Satake parameter $\a_r$.\\

As in the proof of $(3)$, for a given $\lambda\in X_\bullet(A_{r})$, we write
\[
\lambda=\lambda_1\e_1+\lambda_2\e_2+\cdots+\lambda_r\e_r.
\]
Elements $\la\in X_\bullet(A_r)$ with $\la_r=0$ can be view as in $X_\bullet(A_{r-1})$ via the natural inclusion $X_\bullet(A_{r-1})\subset X_\bullet(A_r)$. Now for a fixed $\la\in P^+_{\GL_r}$, we can regard 
$\chi^{\GL_r}_\la(d_{\ul{\a}})$ as a polynomial in $\cS_r$ as $\ul{\a}$ varies. In particular, we can consider 
$\chi^{\GL_r}_\la(d_{\ul{\a}_0})$ provided that it is defined, where $\ul{\a}_0=(\a_1,\a_2,\ldots,\a_{r-1}, 0)$.
Now a lemma in \cite[Section 4]{Ginzburg1990} implies that this is defined when $\la_r\ge 0$, in which case we have
\begin{equation}\label{E:Ginz rel}
\chi^{\GL_r}_{\la}(d_{\ul{\a}_0})
=
\begin{cases}
\chi^{\GL_{r-1}}_{\la}(d_{\ul{\a}'})\quad&\text{if $\la_r=0$},\\
0\quad&\text{if $\la_r>0$},
\end{cases}
\end{equation}
where $\ul{\a}'=(\a_1,\a_2,\ldots,\a_{r-1})$.\\

Following the proof of Ginzburg, we first apply \lmref{L:supp for para vec} and \eqref{E:Psi_r,m}, \eqref{E:Psi_r,m,l}, 
\eqref{E:Whittaker function for GL} to write ($\Re(s)\gg 0$)
\begin{align*}
\Psi^m_{n,r}(v;\a_1,\a_2,\ldots,\a_r; q^{-s+\frac{1}{2}})
=&
\sum_{\la\in P^{+0}_{\GL_r}}
W_v(m_r(\varpi^\la))\chi^{\GL_r}_\la(d_{\ul{\a}})\delta_{B_r}(\varpi^\la)^{-\frac{1}{2}}\nu_r(\varpi^\la)^{s-n+\frac{r}{2}}\\
&+
\sum_{\la\in P^{++}_{\GL_r}}
W_v(m_r(\varpi^\la))\chi^{\GL_r}_\la(d_{\ul{\a}})\delta_{B_r}(\varpi^\la)^{-\frac{1}{2}}\nu_r(\varpi^\la)^{s-n+\frac{r}{2}}
\end{align*}
where 
\[
P^{+0}_{\GL_r}
=
\stt{\la\in P^+_{\GL_r}\mid \la_r=0}
\quad\text{and}\quad
P^{++}_{\GL_r}
=
\stt{\la\in P^+_{\GL_r}\mid \la_r>0}.
\]
At this point, we let $\a_r=0$ in the above equation. This is legal from our discussions in the previous paragraph. 
Then the second summation vanishes according to \eqref{E:Ginz rel}. On the other hand, since  
\[
\delta_{B_r}(\varpi^\la)^{-\frac{1}{2}}\nu_r(\varpi^\la)^{s-n+\frac{r}{2}}
=
\delta_{B_{r-1}}(\varpi^{\la})^{-\frac{1}{2}}\nu_{r-1}(\varpi^{\la})^{s-n+\frac{r-1}{2}}
\]
for $\la\in P^{+0}_{\GL_r}$ and $P^{+0}_{\GL_r}$ is a disjoint union of $P^{++}_{\GL_{r-1}}$ and $P^{+0}_{\GL_{r-1}}$,
the first summation becomes 
\begin{align*}
\Psi^m_{n,r}(v;\a_1,\a_2,\ldots,\a_{r-1}, 0; q^{-s+\frac{1}{2}})
&=
\sum_{\la\in P^{+0}_{\GL_{r-1}}}
W_v(m_{r-1}(\varpi^\la))\chi^{\GL_{r-1}}_\la(d_{\ul{\a}'})
\delta_{B_{r-1}}(\varpi^\la)^{-\frac{1}{2}}\nu_{r-1}(\varpi^\la)^{s-n+\frac{r-1}{2}}\\
&+
\sum_{\la\in P^{++}_{\GL_{r-1}}}
W_v(m_{r-1}(\varpi^\la))\chi^{\GL_{r-1}}_\la(d_{\ul{\a}'})
\delta_{B_{r-1}}(\varpi^\la)^{-\frac{1}{2}}\nu_{r-1}(\varpi^\la)^{s-n+\frac{r-1}{2}}\\
&=
\Psi^m_{n,r-1}(v;\a_1,\a_2,\ldots,\a_{r-1}; q^{-s+\frac{1}{2}})
\end{align*}
by using \eqref{E:Ginz rel} again. This implies \eqref{E:r to r-1}.\qed

\subsubsection{Proof of $(5)$}
By \eqref{E:Xi_r,m} and the definition of $\Xi^m_{n,n}(v;X_1,X_2,\ldots,X_n)$, it is enough to show
\begin{equation}\label{E:satake and jacquet}
\Psi^m_{n,n}(\varphi\star v;X_1,X_2,\ldots,X_n;Y)
=
\sS^0_{n,m}(\varphi)(YX_1,YX_2,\ldots,YX_n)
\Psi^m_{n,n}(v;X_1,X_2,\ldots,X_n;Y)
\end{equation}
for $\varphi\in\cH(H_n(F)//R_{n,m})$ and $v\in\cV_\pi^{R_{n,m}}$. Let $a\in\GL_n(F)$. We first compute
\begin{align}\label{E:Whittaker and satake}
\begin{split}
W_{\varphi\star v}(m_n(a))
&=
\int_{H_n(F)}
W_v(m_n(a)h^{-1})\varphi(h)dh
=
\int_{Q_n(F)}
W_{v}(m_n(a)p^{-1})\varphi(p)\delta_{Q_n}(p)dp\\
&=
\int_{\GL_n(F)}
W_v(m_n(ab^{-1}))\delta^{\frac{1}{2}}_{Q_n}(m_n(b))
\left(\int_{N_n(F)}\varphi(m_n(b)u)\delta^{\frac{1}{2}}_{Q_n}(m_n(b))du\right)db\\
&=
\int_{\GL_n(F)}
W_v(m_n(ab^{-1}))\delta^{\frac{1}{2}}_{Q_n}(m_n(b))\imath_{r,m}(\varphi)(b)db.
\end{split}
\end{align}
Recall that $\imath_{n,m}(\varphi)\in\cH(\GL_n(F)//\GL_n(\frak{o}))$ is the element defined by \eqref{E:middle satake}.
Next we have 
\begin{align}\label{E:satake for Whittaker}
\begin{split}
\int_{\GL_n(F)}
&\imath_{n,m}(\varphi)(b)W(ab; q^{-s+\frac{1}{2}}\a_1, q^{-s+\frac{1}{2}}\a_2,\ldots, q^{-s+\frac{1}{2}}\a_n)db\\
&=
\sS^0_{n,m}(\varphi)(q^{-s+\frac{1}{2}}\a_1, q^{-s+\frac{1}{2}}\a_2,\ldots, q^{-s+\frac{1}{2}}\a_n)
W(a; q^{-s+\frac{1}{2}}\a_1, q^{-s+\frac{1}{2}}\a_2,\ldots, q^{-s+\frac{1}{2}}\a_n)
\end{split}
\end{align}
by \eqref{E:satake for GL} and the commutative diagram \eqref{E:diagram}. Now by \eqref{E:Psi_r,m}, we can write 
\begin{align*}
\Psi^m_{n,n}&(\varphi\star v; \a_1, \a_2,\ldots, \a_n; q^{-s+\frac{1}{2}})
=
\int_{Z_n(F)\backslash\GL_n(F)}
W_{\varphi\star v}(m_n(a))W(a; q^{-s+\frac{1}{2}}\a_1, q^{-s+\frac{1}{2}}\a_2,\ldots, q^{-s+\frac{1}{2}}\a_n)da.
\end{align*}
Applying \eqref{E:Whittaker and satake}, the right hand side becomes 
\begin{align*}
\int_{Z_n(F)\backslash\GL_n(F)}
\int_{\GL_n(F)}W_v(m_n(ab^{-1}))\imath_{n,m}(\varphi)(b)
W(a; q^{-s+\frac{1}{2}}\a_1, q^{-s+\frac{1}{2}}\a_2,\ldots, q^{-s+\frac{1}{2}}\a_n)
\delta_{Q_n}^{-\frac{1}{2}}(m_n(ab^{-1}))dbda.
\end{align*}
After changing the variable $a\mapsto ab$ and taking \eqref{E:satake for Whittaker} into account, we obtain
\begin{align*}
\int_{Z_n(F)\backslash\GL_n(F)}&
W_v(m_n(a))
\left(\int_{\GL_n(F)}
\imath_{n,m}(\varphi)(b)W(ab; q^{-s+\frac{1}{2}}\a_1, q^{-s+\frac{1}{2}}\a_2,\ldots, q^{-s+\frac{1}{2}}\a_n)db\right)
\delta_{Q_n}^{-\frac{1}{2}}(m_n(a))da\\
&=
\sS^0_{n,m}(\varphi)(q^{-s+\frac{1}{2}}\a_1, q^{-s+\frac{1}{2}}\a_2,\ldots, q^{-s+\frac{1}{2}}\a_n)
\Psi^m_{n,n}(v; \a_1, \a_2,\ldots, \a_n; q^{-s+\frac{1}{2}}).
\end{align*}
This proves \eqref{E:satake and jacquet}.\qed\\

Our proof of \propref{P:main prop} is now complete.\qed

\section{Proof of the Main Results}\label{S:prove}
We prove our main results in this section. For this, we follow the notation and conventions in \S\ref{SS:convention}.
The proof of \thmref{T:main} will be given in \S\ref{SS:RS int attached to newform} and 
\S\ref{SS:RS int attached to oldform}, under the following hypothesis:
\[
\text{$The$ $space$ $\cV_{\pi}^{K_{n,a_\pi}}$ $is$ $one$-$dimensional$ $and$
$\Lambda_{\pi,\psi}$ $is$ $nontrivial$ $on$ $\cV_{\pi}^{K_{n,a_\pi}}$.} 
\]
So in the first two subsections, we fix an element $v_\pi\in\cV_\pi^{K_{n,a_\pi}}$ with $\Lambda_{\pi,\psi}(v_\pi)=1$. 
In \S\ref{SS:n=r=2}, on the other hand, we will prove \thmref{T:main'}, and we will conclude with a comparison between 
bases in \S\ref{SSS:concluding remark}.

\subsection{Rankin-Selberg integrals attached to newforms}\label{SS:RS int attached to newform}
In this subsection, we prove \thmref{T:main} except the inequality \eqref{E:dim bd} whose proof will be postponed to the 
next subsection.

\subsubsection*{Proof of Conjecture \ref{C1} $(2)$ and \eqref{E:main eqn}}
Let $1\le r\le n$ and $u_{n,r,a_\pi}\in J_{n,a_\pi}$ be the elements given by \eqref{E:u_r,m}. Note that 
$u_{n,1,a_\pi}\in K_{n,a_\pi}=J_{n,a_\pi}$ if $a_\pi=0$, while $u_{n,1,a_\pi}\in J_{n,a_\pi}\setminus K_{n,a_\pi}$ if 
$a_\pi>0$. Let $\e_r\in\stt{\pm 1}$ such that $\pi(u_{n,r,a_\pi})v_\pi=\e_r v_\pi$. Certainly, we have $\e_r=1$ for all $r$ if 
$a_\pi=0$. We need to show $\e_1=\e_\pi$ and \eqref{E:main eqn}. For this, it suffices to prove
\begin{equation}\label{E:main identity}
\Xi^{a_\pi}_{n,n}(v_\pi; X_1,X_2,\ldots,X_n)=1.
\end{equation}
Indeed, if \eqref{E:main identity} is valid, then \propref{P:main prop} $(1)$ and $(4)$ (with $m=a_\pi$) will give us 
\eqref{E:main eqn}. On the other hand, by \eqref{E:main identity} and \propref{P:main prop} $(2)$, we have 
$\Xi_{n,1}^{a_\pi}(v_\pi; X^{-1}_1)=\Xi_{n,1}^{a_\pi}(v_\pi; X_1)=1$ and hence by \propref{P:main prop} $(2)$
\[
\e_1
=\e_1\Xi_{n,1}^{a_\pi}(v_\pi; X^{-1}_1)
=\Xi_{n,1}^{a_\pi}(\pi(u_{n,1,a_\pi})v_\pi; X^{-1}_1)
=\e_\pi\Xi^{a_\pi}_{n,1}(v_\pi;X_1)
=\e_\pi.
\]
So the proof reduces to verifying \eqref{E:main identity}.\\

To prove \eqref{E:main identity}, we first claim that 
$\Xi^{a_\pi}_{n,n}(v_\pi; X_1, X_2, \ldots, X_n)$ is a constant. Let $Y_j$ be the $j$-th elementary symmetric 
polynomial in $X_1, X_2,\ldots, X_n$ for $1\le j\le n$. Then we have 
\[
\cS_n=\bbC[Y_1, Y_2, \ldots, Y_{n-1}, Y_n, Y_n^{-1}].
\] 
Note that $Y_n=X_1X_2\cdots X_n$ gives a $\bbZ$-grading of $\cS_n=\oplus_{\ell\in\bbZ}\cS_{n,\ell}$ by the degree 
of $Y_n$, i.e. $\cS_{n,\ell}=\bbC[Y_1, Y_2,\ldots, Y_{n-1}]\, Y_n^\ell$ for $\ell\in\bbZ$. Now by \lmref{L:supp for para vec}, 
\eqref{E:Psi_r,m,l}, \eqref{E:Xi_r,m,l} and the observation
\[
{\rm deg}_{Y_n}W(\varpi^\la;X_1,X_2,\ldots, X_n;\b{\psi})=\la_n
\]
where $\la=\la_1\e_1+\la_2\e_2+\cdots+\la_n\e_n\in P^+_{\GL_n}$, we find that
\begin{equation}\label{E:degree in Y_n}
\Xi^m_{n,n}(v;X_1, X_2,\ldots, X_n)\in\bigoplus_{\ell\ge 0}\cS_{n,\ell}
\end{equation}
for every $v\in\cV_\pi^{K_{n,m}}$ and $m\ge 0$. Then \propref{P:main prop} $(2)$ gives
\[
\e_n\Xi^{a_\pi}_{n,n}(v_\pi;X^{-1}, X_2^{-1},\ldots, X_n^{-1})
=
\Xi^{a_\pi}_{n,n}(\pi(u_{n,n,a_\pi})v_\pi;X^{-1}, X_2^{-1},\ldots, X_n^{-1})
=
\e_\pi^n
\Xi^{a_\pi}_{n,n}(v_\pi; X_1, X_2, \ldots, X_n).
\]
Combining this with \eqref{E:degree in Y_n} (with $m=a_\pi$), we see that 
$\Xi^{a_\pi}_{n,n}(v_\pi; X_1, X_2, \ldots, X_n)=c$ for some constant $c$.\\

To complete the proof, it remains to show $c=1$.  By \propref{P:main prop} $(1)$ and $(4)$, Remark \ref{R:zeta integral} 
and the normalization of the spherical Whittaker functions, we find that
\[
c
=
\Xi^{a_\pi}_{n,n}(v_\pi; q^{-s+\frac{1}{2}},\underbrace{0,0,\ldots,0}_{n-1})
=
\Xi^{a_\pi}_{n,1}(v_\pi; q^{-s+\frac{1}{2}})
=
\frac{Z(s,v_\pi)}{L(s,\phi_\pi)}
\]
which gives
\[
c\,L(s,\phi_\pi)
=
Z(s,v_\pi).
\]
By writing both sides as power series in $q^{-s}$, we see that the constant term in the LHS is just $c$. 
On the other hand, by \lmref{L:supp for para vec} and \lmref{L:exp of RS int},  the constant term in the RHS is 
$\Lambda_{\pi,\psi}(v_\pi)=1$. This finishes the proof.\qed

\subsection{Rankin-Selberg integrals attached to oldforms}\label{SS:RS int attached to oldform}
We prove \eqref{E:dim bd} in this subsection. The proof is achieved by computing the Rankin-Selberg integrals 
attached to the subsets $\cB_{\pi,m}$ defined in \S\ref{SS:conj basis}, which is also one of the main objectives of this 
work (cf. \thmref{T:main for oldform}). However, we should mention that our results for 
oldforms are less explicit, as they involve two positive constants $q_n$ and $q'_n$ 
\footnote{It seems that these constants only depend on $q$ and $n$.}
implicitly given by 
\begin{equation}\label{E:q_n}
Z(s,\theta(v_\pi))=q_nq^{-s+\frac{1}{2}}Z(s,v_\pi)
\quad\text{and}\quad
Z(s,\theta'(v_\pi))=q'_nZ(s,v_\pi)
\end{equation}
where $\theta=\theta_{a_\pi+1}, \theta'=\theta'_{a_\pi+1}$ are level raising operators defined in \S\ref{SSS:level raising}. 
This can be proved by following the arguments in \cite[Proposition 9.1.3]{Tsai2013} (without the Hypothesis \ref{H}). 
We should point out that $q_2=q'_2=q$ by a result of Roberts-Schmidt (\cite[Proposition 4.1.3]{RobertsSchmidt2007}). 
Here we will not try to compute these numbers explicitly (for general $n$), but only show that they are equal 
(under the hypothesis mentioned in the beginning of this section).\\

We start with two lemmas. 
 
\begin{lm}\label{L:eta action}
Let $v\in\cV_\pi^{R_{n,m}}$. We have 
\[
\Xi^{m+2}_{n,n}(\eta(v);X_1,X_2,\ldots, X_n)
=
q^{\frac{n(n-1)}{2}}(X_1X_2\cdots X_n)\Xi^m_{n,n}(v;X_1,X_2,\ldots, X_n).
\]
\end{lm}

\begin{proof}
Recall that $\eta=\pi(\varpi^{-\mu_n})$ and we have $\eta(v)\in\cV_\pi^{R_{n,m+2}}$ by \eqref{E:R*}. 
After changing the variable $a\mapsto a(\varpi I_n)$ in \eqref{E:Psi_r,m,l} and then using the fact 
\[
W(a(\varpi I_n);X_1,X_2,\ldots, X_n; \b{\psi})
=
(X_1X_2,\cdots X_n)
W(a;X_1,X_2,\ldots, X_n; \b{\psi})
\]
%for $a\in\GL_n(F)$, 
we find that
\[
\Psi^{m+2}_{n,n,\ell}(\eta(v);X_1,X_2,\ldots, X_n)
=
q^{\frac{n(n-1)}{2}}(X_1X_2\cdots X_n)\Psi^m_{n,n,\ell-n}(v;X_1,X_2,\ldots, X_n)
\]
for all $\ell\in\bbZ$. It follows that 
\[
\Psi^{m+2}_{n,n}(\eta(v);X_1,X_2,\ldots, X_n;Y)
=
q^{\frac{n(n-1)}{2}}(X_1X_2\cdots X_n)Y^n \Psi^m_{n,n}(v;X_1,X_2,\ldots, X_n; Y)
\]
and hence
\[
\Xi^{m+2}_{n,n}(\eta(v);X_1,X_2,\ldots, X_n;Y)
=
q^{\frac{n(n-1)}{2}}(X_1X_2\cdots X_n)Y^n \Xi^m_{n,n}(v;X_1,X_2,\ldots, X_n; Y)
\]
by \eqref{E:Xi_r,m}. Now the lemma follows from letting $Y=1$ in both sides.
\end{proof}

In the following lemma, for a given $1\le r\le n$, we denote by $Y_j(X_1, X_2,\ldots, X_r)$ the $j$-th elementary 
symmetric polynomial in $X_1, X_2,\ldots, X_r$ for $1\le j\le r$ and put $Y_0(X_1, X_2,\ldots, X_r)=1$. 
Then one has the relations 
\[
Y_j(X_1,X_2,\ldots, X_{r-1},0)=Y_j(X_1,X_2,\ldots, X_{r-1})
\quad\text{and}\quad
Y_r(X_1,X_2,\ldots, X_{r-1},0)=0
\]
for $0\le j\le r-1$.

\begin{lm}\label{L:level a+1}
We have $q_n=q_n'$ and
\[
\Xi^{a_\pi+1}_{n,n}(\theta(v_\pi);X_1, X_2,\ldots, X_n)
=
q_n\sum^n_{j=0}\left(\frac{1+(-1)^{j+1}}{2}\right)Y_j(X_1,X_2,\ldots,X_n)
\]
and 
\[
\Xi^{a_\pi+1}_{n,n}(\theta'(v_\pi);X_1, X_2,\ldots, X_n)
=
q'_n\sum^n_{j=0}\left(\frac{1+(-1)^{j}}{2}\right)Y_j(X_1,X_2,\ldots, X_n).
\]
\end{lm}

\begin{proof}
We first show that if $v\in\cV_\pi^{K_{n,a_\pi+1}}$ and $\pi(u_{n,1,a_\pi+1})v=\e v$ for some $\e\in\stt{\pm 1}$, then 
there exists $b_0\in\bbC$ such that 
\begin{equation}\label{E:Xi with level a+1}
\Xi^{a_\pi+1}_{n,n}(v;X_1,X_2,\ldots, X_n)
=
b_0\sum_{j=0}^n(\e\e_\pi)^j Y_j(X_1,X_2,\ldots, X_n).
\end{equation}
Note that the assumption $\pi(u_{n,1,a_\pi+1})v=\e v$ implies 
\begin{equation}\label{E:ev}
\pi(u_{n,r,a_\pi+1})v=\e ^r v
\end{equation}
for $1\le r\le n$, where $u_{n,r,a_\pi+1}\in J_{n,a_\pi+1}$ are the elements given by \eqref{E:u_r,m}.
Indeed, if we put
\[
s_r
=
\begin{pmatrix}
&I_{r-1}&\\
1&&\\
&&I_{n-r}
\end{pmatrix}
\quad\text{and}\quad
s'_r
=
\begin{pmatrix}
(-1)^{r-1}&\\
&-I_{r-1}&\\
&&I_{n-r}
\end{pmatrix}
\]
(in $\GL_n(\frak{o})$), then we have
\[
u_{n,r-1,m}u_{n,r,m}=m_{r-1}(\jmath_{r-1})m_r(s'_r)m_r(s_r)u_{n,1,m}m_r(s_r)^{-1}m_r(\jmath_{r})
\]
for $m\ge 0$. Since $m_r(s_r)$, $m_r(s'_r)$ $m_{r-1}(\jmath_{r-1})$ and $m_r(\jmath_r)$ are contained in $K_{n,m}$, 
we deduce that
\[
\pi(u_{n,r-1,a_\pi+1}u_{n,r,a_\pi+1})v=\e v.
\]
This implies \eqref{E:ev}.\\

Now by \eqref{E:degree in Y_n}, we may assume 
\begin{equation}\label{E:rought Xi for level a+1}
\Xi^{a_\pi+1}_{n,n}(v;X_1,X_2,\ldots, X_n)
=
\sum_{\ul{\ell}}
b_{\ul{\ell}}
Y_1(X_1,X_2,\ldots, X_n)^{\ell_1}Y_2(X_1,X_2,\ldots, X_n)^{\ell_2}\cdots Y_n(X_1,X_2,\ldots, X_n)^{\ell_n}
\end{equation}
for some $\ul{\ell}=(\ell_1,\ell_2,\ldots, \ell_n)\in\bbZ^n_{\ge 0}$ and $b_{\ul{\ell}}\in\bbC$ with $b_{\ul{\ell}}=0$ for 
almost all $\ul{\ell}$. Since 
\[
Y_j(X^{-1}_1,X_2^{-1}, \ldots, X_r^{-1})=Y_{r-j}(X_1,X_2,\ldots, X_r)Y_r(X_1,X_2,\ldots, X_r)^{-1}
\]
for $1\le j\le r-1$, the functional equation in \propref{P:main prop} $(2)$ and the relation \eqref{E:ev} give
\begin{align*}
\Xi^{a_\pi+1}_{n,n}&(v;X_1,X_2,\ldots, X_n)\\
&=
\sum_{\ul{\ell}}
(\e\e_\pi)^n b_{\ul{\ell}}
Y_1(X_1,X_2,\ldots, X_n)^{\ell_{n-1}}Y_2(X_1,X_2,\ldots, X_n)^{\ell_{n-2}}\cdots 
Y_n(X_1,X_2,\ldots, X_n)^{1-\ell_1-\ell_2-\cdots-\ell_n}.
\end{align*}
Since $1-\ell_1-\ell_2-\cdots-\ell_n\ge 0$, we find that $\ell_j\le 1$ for $1\le j\le n$. Hence 
\eqref{E:rought Xi for level a+1} becomes
\[
\Xi^{a_\pi+1}_{n,n}(v;X_1,X_2,\ldots, X_n)
=
\sum_{j=0}^n b_j Y_j(X_1,X_2,\ldots, X_n).
\]
To prove \eqref{E:Xi with level a+1}, it remains to show 
\begin{equation}\label{E:c_j for a+1}
b_j=b_0(\e\e_\pi)^j
\end{equation}
for $1\le j\le n$. For this, we apply \propref{P:main prop} (4) to get
\[
\Xi^{a_\pi+1}_{n,r}(v;X_1,X_2,\ldots, X_r)
=
\sum_{j=0}^r b_j Y_j(X_1,X_2,\ldots, X_r)
\]
for $1\le r\le n$. Then by the functional equation for $\Xi^{a_\pi+1}_{n,r}(v;X_1,X_2,\ldots, X_r)$ and \eqref{E:ev}, one 
obtains \eqref{E:c_j for a+1} (when $j=r$) after comparing the constant terms. This shows \eqref{E:Xi with level a+1}. \\

Now let us put 
\[
v_\pm
=
\theta(v_\pi)\pm\theta'(v_\pi).
\]
Then we have $\pi(u_{n,1,a_\pi+1})v_\pm=\pm\e_\pi v_\pm$ by the fact that 
$\pi(u_{n,1,a_\pi})v_\pi=\e_\pi v_\pi$ and the definitions of $\theta, \theta'$ (cf. \eqref{E:theta}). Therefore, there exist 
$b_0^\pm\in\bbC$ such that 
\[
\Xi^{a_\pi+1}_{n,n}(v_\pm;X_1,X_2,\ldots, X_n)
=
b^\pm_0\sum_{j=0}^n(\pm 1)^j Y_j(X_1,X_2,\ldots, X_n)
\]
by \eqref{E:Xi with level a+1}. We then deduce that
\begin{equation}\label{E:rought Xi theta}
\Xi^{a_\pi+1}_{n,n}(\theta(v_\pi);X_1,X_2,\ldots, X_n)
=
\sum_{j=0}^n\left(\frac{b^+_0+b_0^-(-1)^j}{2}\right)Y_j(X_1,X_2,\ldots, X_n)
\end{equation}
and 
\begin{equation}\label{E:rought Xi theta'}
\Xi^{a_\pi+1}_{n,n}(\theta'(v_\pi);X_1,X_2,\ldots, X_n)
=
\sum_{j=0}^n\left(\frac{b^+_0-b_0^-(-1)^j}{2}\right)Y_j(X_1,X_2,\ldots, X_n).
\end{equation}
To complete the proof, we need to solve $b_0^\pm$ and show that $q_n=q'_n$. 
By \propref{P:main prop} $(4)$, Remark \ref{R:zeta integral} and \eqref{E:q_n}, \eqref{E:main identity}, we first find that
\[
\Xi^{a_\pi+1}_{n,n}(\theta(v_\pi);q^{-s+\frac{1}{2}},\underbrace{0,\ldots, 0}_{n-1})
=
\Xi^{a_\pi+1}_{n,1}(\theta(v_\pi);q^{-s+\frac{1}{2}})
=
\frac{Z(s,\theta(v_\pi))}{L(s,\phi_\pi)}
=
q_nq^{-s+\frac{1}{2}}
\]
and  
\[
\Xi^{a_\pi+1}_{n,n}(\theta'(v_\pi);q^{-s+\frac{1}{2}},\underbrace{0,\ldots, 0}_{n-1})
=
\Xi^{a_\pi+1}_{n,1}(\theta'(v_\pi);q^{-s+\frac{1}{2}})
=
\frac{Z(s,\theta'(v_\pi))}{L(s,\phi_\pi)}
=
q'_n.
\]
On the other hand, \eqref{E:rought Xi theta} and \eqref{E:rought Xi theta'} imply 
\[
\Xi^{a_\pi+1}_{n,n}(\theta(v_\pi);q^{-s+\frac{1}{2}},\underbrace{0,\ldots, 0}_{n-1})
=
\left(\frac{b^+_0+b^-_0}{2}\right)+\left(\frac{b^+_0-b^-_0}{2}\right)q^{-s+\frac{1}{2}}
\]
and 
\[
\Xi^{a_\pi+1}_{n,n}(\theta'(v_\pi);q^{-s+\frac{1}{2}},\underbrace{0,\ldots, 0}_{n-1})
=
\left(\frac{b^+_0-b^-_0}{2}\right)+\left(\frac{b^+_0+b^-_0}{2}\right)q^{-s+\frac{1}{2}}.
\]
By comparing the coefficients, we obtain
\[
b^+_0-b_0^-
=
2q_n
=
2q'_n
\quad\text{and}\quad
b^+_0+b^-_0=0.
\]
It follows that $b_0^\pm=\pm q_n= \pm q'_n$ as desired.
\end{proof}

Let's summarize our computations in the following theorem.

\begin{thm}\label{T:main for oldform}
Let  $\cB_{\pi,m}$ be the subsets defined in \S\ref{SS:conj basis}. Then the Rankin-Selberg integrals 
$\Psi_{n,r}(v\ot\xi^m_{\tau,s})$ attached to $v\in\cB_{\pi,m}$ and $\xi^m_{\tau,s}$ can be computed by using 
\propref{P:main prop} $(1)$, $(4)$ and $(5)$ together with \lmref{L:eta action}, \lmref{L:level a+1} and 
\eqref{E:main identity}.
\end{thm}

Now we can prove \eqref{E:dim bd}.

\subsubsection*{Proof of \eqref{E:dim bd}}
We are going to show that the subsets $\cB_{\pi,m}$ are linearly independent. Then since the cardinality of $\cB_{\pi,m}$
is given by the RHS of \eqref{E:dim bd}, the proof follows. The case when $m$ and $a_\pi$ have the same parity is much 
easier and is simply a consequence of \propref{P:main prop} $(5)$, \lmref{L:eta action} and \eqref{E:main identity}. 
Suppose that $m$ and $a_\pi$ have the opposite parity. We first claim 
\[
\sS^0_{n,m}(\varphi_{\t{\la},m})(X_1,X_2,\ldots,X_{n-1}, X_n)
=
\sS^0_{n,m}(\varphi_{\la,m})(X_1, X_2,\ldots, X_{n-1},X_n^{-1})
\]
for $\la\in X_\bullet(T_n)$. Recall that $\t{\la}\in X_\bullet(T_n)$ is defied in \S\ref{SS:conj basis}. For this, let 
\[
u_m=m_{n-1}(\jmath_{n-1})u_{n,n-1,m}u_{n,n,m}m_n(\jmath_n)\in J_{n,m}.
\]
Then we have $u_{n,m}^{-1}R_{n,m}u_m=R_{n,m}$ by \lmref{L:general AL elt} and the fact that both $m_n(\jmath_n)$ and 
$m_{n-1}(\jmath_{n-1})$ are contained in $R_{n,m}$. It follows that $\varphi^{u_m}_{\la,m}=\varphi_{\t{\la},m}$, where
$\varphi^{u_m}_{\la,m}\in\cH(H_n(F)//R_{n,m})$ is defined by $\varphi^{u_m}_{\pi,m}(h)=\varphi_{\la,m}(u_m^{-1}hu_m)$.   
Since (using \eqref{E:satake for H}) 
\[
\varsigma^0_{n,m}(\varphi_{\t{\la},m})(t)
=\varsigma^0_{n,m}(\varphi^{u_m}_{\la,m})(t)
=\varsigma^0_{n,m}(\varphi_{\la,m})(u_m^{-1}t u_m)
\]
for $t\in T_n(F)$, the claim follows.\\

Now suppose that 
\[
\sum_{\la\in P^+_{G_n}}
c_\la \eta^\square_{\la,a_\pi+1,m}\circ \theta(v_\pi)
=
\sum_{\la\in P^+_{G_n}}
c'_\la \eta^\square_{\la,a_\pi+1,m}\circ \theta'(v_\pi)
\]
for some $c_\la, c'_\la\in\bbC$. Then by \propref{P:main prop} $(5)$ and \lmref{L:eta action}, we obtain
\begin{align}\label{E:LI}
\begin{split}
\sum_{\la\in P^+_{G_n}}&c_\la f_\la(X_1,X_2,\ldots, X_n)\Xi^{a_\pi+1}_{n,n}(\theta(v_\pi); X_1,X_2,\ldots, X_n)\\
&=
\sum_{\la\in P^+_{G_n}}c'_\la f_\la(X_1,X_2,\ldots, X_n)\Xi^{a_\pi+1}_{n,n}(\theta'(v_\pi);X_1,X_2,\ldots, X_n)
\end{split}
\end{align}
after removing some nonzero functions from both sides, where we put
\[
f_\la(X_1,X_2,\ldots, X_n)
:=
\sS^0_{n,a_\pi+1}(\varphi_{\la,a_\pi+1})(X_1,X_2,\ldots, X_n)
+
\sS^0_{n,a_\pi+1}(\varphi_{\t{\la},a_\pi+1})(X_1,X_2,\ldots, X_n)
\]
for $\la\in P^+_{G_n}$. From the claim in the previous paragraph, we know that 
\[
f_\la(X_1,X_2,\ldots,X^{-1}_n)=f_\la(X_1,X_2,\ldots, X_n).
\] 
On the other hand, by \lmref{L:level a+1}, we have
\[
\Xi^{a_\pi+1}_{n,n}(\theta(v_\pi); X_1,X_2,\ldots, X^{-1}_n)
=
X_n^{-1}\,\Xi^{a_\pi+1}_{n,n}(\theta'(v_\pi); X_1,X_2,\ldots, X_n)
\]
and 
\[
\Xi^{a_\pi+1}_{n,n}(\theta'(v_\pi); X_1,X_2,\ldots, X^{-1}_n)
=
X_n^{-1}\,\Xi^{a_\pi+1}_{n,n}(\theta(v_\pi); X_1,X_2,\ldots, X_n).
\]
Thus if we replace $X_n$ with $X_n^{-1}$ in \eqref{E:LI}, we get 
\begin{align}\label{E:LI2}
\begin{split}
\sum_{\la\in P^+_{G_n}}&c_\la f_\la(X_1,X_2,\ldots, X_n)\Xi^{a_\pi+1}_{n,n}(\theta'(v_\pi); X_1,X_2,\ldots, X_n)\\
&=
\sum_{\la\in P^+_{G_n}}c'_\la f_\la(X_1,X_2,\ldots,X_n)\Xi^{a_\pi+1}_{n,n}(\theta(v_\pi); X_1,X_2,\ldots, X_n)
\end{split}
\end{align}
after dividing $X^{-1}_n$ from both sides. Since 
\[
\Xi^{a_\pi+1}_{n,n}(\theta(v_\pi); X_1,X_2,\ldots, X_n)
\pm
\Xi^{a_\pi+1}_{n,n}(\theta'(v_\pi); X_1,X_2,\ldots, X_n)
\]
are nonzero (again by \lmref{L:level a+1}), we can use \eqref{E:LI} and \eqref{E:LI2} to deduce
\[
\sum_{\la\in P^+_{G_n}}c_\la f_\la(X_1,X_2,\ldots, X_n)
=
\sum_{\la\in P^+_{G_n}}c'_\la f_\la(X_1,X_2,\ldots, X_n)
=
0.
\]
This implies $c_\la=c'_\la=0$ for all $\la\in P^+_{G_n}$ and hence concludes the proof.\qed\\

This also finishes the proof of \thmref{T:main}.\qed

\subsection{Case when $n=r=2$}\label{SS:n=r=2}
This subsection is devoted to prove \thmref{T:main'}. Here we don't need the hypothesis mentioned in the beginning of 
this section. We begin with recalling the accidental isomorphism $\vartheta: \SO_5\cong{\rm PGSp}_4$ and some 
formulae for the spherical Whittaker function of $\tau$.

\subsubsection{Accidental isomorphism}
Let $\bbF$ be a field with characteristic different from $2$. Let $(W,\langle,\rangle)$ be the $4$-dimensional symplectic 
space over $\bbF$. Fix an ordered basis $\stt{w_1, w_2, w^*_2, w^*_1}$ of $W$ so that the associated Gram matrix is 
given by
\[
\pMX{}{\jmath_2}{-\jmath_2}{}.
\]
Let $(\t{V}, (,))$ be the $6$-dimensional quadratic space over $\bbF$ with $\t{V}=\bigwedge^2 W$ and the symmetric 
bilinear form $(,)$ defined by 
\[
v_1\wedge v_2=(v_1,v_2)(w_1\wedge w_2\wedge w^*_1\wedge w^*_2)
\]
for $v_1,v_2\in\t{V}$. Let $\t{v}=w_1\wedge w^*_1+w_2\wedge w^*_2\in\t{V}$ and put 
\[
V
=
\stt{v\in\t{V}\mid (v,\t{v})=0}.
\]
Denote also by $(,)$ the restriction of the symmetric bilinear form to $V$ so that $(V,(,))$ becomes a $5$-dimensional 
quadratic space over $\bbF$. Its Gram matrix associated to the ordered basis $\stt{e_1, e_2, v_0, f_2, f_1}$ 
is given by \eqref{E:S} with $n=2$, where
\[
e_1=w_1\wedge w_2,\,\, e_2=w_1\wedge w^*_2,\,\ v_0=w_1\wedge w^*_1-w_2\wedge w_2^*,\,\,f_2=w_2\wedge w_1^*,
\,\,f_1=w_1^*\wedge w_2^*.
\]
Let $\t{\vartheta}:{\rm GSp}(W)\to\SO(\t{V})$ be the homomorphism defined by 
\[
\t{\vartheta}(h)=\mu(h)^{-1}\Wedge^2 h
\]
for $h\in {\rm GSp}(W)$. Here $\mu:{\rm GSp}(W)\to\bbF^\x$ is the similitude character. Then because 
$\t{\vartheta}(h)\t{v}=\t{v}$, the homomorphism induces an exact sequence 
\[
1\longto\bbF^\x\overset{\iota}{\longto}{\rm GSp}(W)\overset{\vartheta}{\longto}\SO(V)\longto 1
\]
where $\iota(a)=aI_{W}$ with $I_W:W\to W$ the identity map. We continue to use $\vartheta$ to denote the isomorphism
$\vartheta:{\rm PGSp}(W)\overset{\sim}{\to}\SO(V)$ induced from the exact sequence. Using the ordered bases
$\stt{w_1,w_2, w_2^*, w_1^*}$ and $\stt{e_1,e_2,v_0,f_2, f_1}$, we can identify (when $\bbF=F$)
${\rm PGSp}(W)\cong{\rm PGSp}_4(F)$ and $\SO(V)\cong\SO_5(F)$. Then under $\vartheta$, the paramodular 
subgroups $K(\frak{p}^m)$ defined in \cite{RobertsSchmidt2007} map onto $K_{2,m}$ defined in 
\S\ref{SS:paramodular subgroup}. As a consequence, we can transfer results of 
Roberts-Schmidt to that of $\SO_5(F)$ via $\vartheta$.

\subsubsection{Some formulae}
For a given $\la\in X_\bullet(A_2)$, we denote $\la=\la_1\e_1+\la_2\e_2$.
Then we have 
\begin{equation}\label{E:Whittaker for GL_2}
W(\varpi^\la;\a_1,\a_2;\b{\psi})
=
q^{-\frac{1}{2}(\la_1-\la_2)}\sum_{j=0}^{\la_1-\la_2}\a_1^{\la_2+j}\a_2^{\la_1-j}
\end{equation}
if $\la_1\ge \la_2$, and $0$ otherwise (\cite[Section 2.4]{Schmidt2002}, \cite{Shintani1976}).
On the other hand, for a given $y\in F$, let us define $z_2(y), n_2(y)\in \SO_5(F)$ to be
\[
z_2(y)
=
\begin{pmatrix}
1&y&\\
&1&\\
&&1&\\
&&&1&-y\\
&&&&1
\end{pmatrix}
\quad\text{and}\quad
n_2(y)
=
\begin{pmatrix}
1&&&-y&\\
&1&&&y\\
&&1\\
&&&1\\
&&&&1
\end{pmatrix}.
\]
Then via $\vartheta$, equations in \cite[Lemma 3.2.2]{RobertsSchmidt2007} become
\begin{equation}\label{E:exp theta}
\theta(v)
=
\pi(m_2(\varpi^{-\e_1}))v
+
\sum_{c\in\frak{f}}
\pi(m_2(\varpi^{-\e_2})z_2(c\varpi^{-1}))v
\end{equation}
and
\begin{equation}\label{E:exp theta'}
\theta'(v)
=
\pi(m_2(\varpi^{-\e_1-\e_2}))v
+
\sum_{c\in\frak{f}}
\pi(n_2(c\varpi^{-m-1}))v
\end{equation}
for $v\in\cV_\pi^{K_{2,m}}$ with $m\ge 0$. 

\subsubsection{Proof of \eqref{E:raising theta}}
Let us write $\theta(v)=v_1+v_2$ with $v_1=\pi(m_2(\varpi^{-\e_1}))v$.  By \lmref{L:supp for para vec},
\lmref{L:exp of RS int} and the fact $R_{2,m}\cap M_2(F)=M_2(\frak{o})$, we get 
\begin{align*}
\Psi_{2,2}(\theta(v)\ot\xi^{m+1}_{\tau,s})
=&
\sum_{\la_1\ge \la_2\ge 0} 
W_{\theta(v)}(m_2(\varpi^{\la_1\e_1+\la_2\e_2}))
W(\varpi^{\la_1\e_1+\la_2\e_2};\a_1,\a_2;\b{\psi})q^{-(\la_1+\la_2)s+2\la_1}\\
=&
\sum_{\la_1\ge \la_2\ge 0} 
W_v(m_2(\varpi^{(\la_1-1)\e_1+\la_2\e_2}))W(\varpi^{\la_1\e_1+\la_2\e_2};\a_1,\a_2;\b{\psi})q^{-(\la_1+\la_2)s+2\la_1}\\
&+
\sum_{\la_1\ge \la_2\ge 0} 
W_{v_2}(m_2(\varpi^{\la_1\e_1+\la_2\e_2}))W(\varpi^{\la_1\e_1+\la_2\e_2};\a_1,\a_2;\b{\psi})
q^{-(\la_1+\la_2)s+2\la_1}
\end{align*}
by the Iwasawa decomposition of $\GL_2(F)$. Since $\ker(\psi)=\frak{o}$, we have
\begin{align*}
W_{v_2}(m_2(\varpi^{\la_1\e_1+\la_2\e_2}))
&=
\sum_{c\in\frak{f}}
W_v(m_2(\varpi^{\la_1\e_1+(\la_2-1)\e_2})z_2(c\varpi^{-1}))\\
&=
\sum_{c\in\frak{f}}
W_v(z_2(c\varpi^{\la_1-\la_2})m_2(\varpi^{\la_1\e_1+(\la_2-1)\e_2}))\\
&=
q W_v(m_2(\varpi^{\la_1\e_1+(\la_2-1)\e_2}))
\end{align*}
for $\la_1\ge \la_2$. On the other hand, a direct computation (using \eqref{E:Whittaker for GL_2}) shows
\[
W(\varpi^{(\la_1+1)\e_1+\la_2\e_2};\a_1,\a_2;\b{\psi})
+
q^{-1}W(\varpi^{\la_1\e_1+(\la_2+1)\e_2};\a_1,\a_2;\b{\psi})
=
q^{-\frac{1}{2}}(\a_1+\a_2)W(\varpi^{\la_1\e_1+\la_2\e_2};\a_1,\a_2;\b{\psi})
\]
again for $\la_1\ge \la_2$. Combining these, we find that 
\begin{align*}
\Psi_{2,2}(\theta(v)\ot\xi^{m+1}_{\tau,s})
=&
\sum_{\la_1\ge \la_2\ge 0} 
W_v(m_2(\varpi^{(\la_1-1)\e_1+\la_2\e_2}))W(\varpi^{\la_1\e_1+\la_2\e_2};\a_1,\a_2;\b{\psi})q^{-(\la_1+\la_2)s+2\la_1}\\
&+
\sum_{\la_1\ge \la_2\ge 0} 
W_{v}(m_2(\varpi^{\la_1\e_1+(\la_2-1)\e_2}))W(\varpi^{\la_1\e_1+\la_2\e_2};\a_1,\a_2;\b{\psi})
q^{-(\la_1+\la_2)s+2\la_1+1}\\
=&
\sum_{\la_1\ge \la_2\ge 0} 
W_v(m_2(\varpi^{\la_1\e_1+\la_2\e_2}))W(\varpi^{(\la_1+1)\e_1+\la_2\e_2};\a_1,\a_2;\b{\psi})
q^{-(\la_1+\la_2+1)s+2(\la_1+1)}\\
&+
\sum_{\la_1\ge \la_2\ge 0} 
W_{v}(m_2(\varpi^{\la_1\e_1+\la_2\e_2}))W(\varpi^{\la_1\e_1+(\la_2+1)\e_2};\a_1,\a_2;\b{\psi})
q^{-(\la_1+\la_2+1)s+2\la_1+1}\\
=&
q^{-s+\frac{3}{2}}(\a_1+\a_2)\sum_{\la_1\ge \la_2\ge 0} 
W_v(m_2(\varpi^{\la_1\e_1+\la_2\e_2}))W(\varpi^{\la_1\e_1+\la_2\e_2};\a_1,\a_2;\b{\psi})q^{-(\la_1+\la_2)s+2\la_1}\\
=&
q^{-s+\frac{3}{2}}(\a_1+\a_2)\Psi_{2,2}(v\ot\xi^m_{\tau,s}).
\end{align*}
This proves \eqref{E:raising theta}.\qed

\subsubsection{Proof of \eqref{E:raising theta'}}
The computations are similar to the previous one. We denote $\theta'(v)=v'_1+v'_2$ with 
$v_1'=\pi(m_2(\varpi^{-\e_1-\e_2}))v$.  Then we have
\begin{align*}
\Psi_{2,2}(\theta'(v)\ot\xi^{m+1}_{\tau,s})
=&
\sum_{\la_1\ge \la_2\ge 0} 
W_{\theta'(v)}(m_2(\varpi^{\la_1\e_1+\la_2\e_2}))
W(\varpi^{\la_1\e_1+\la_2\e_2};\a_1,\a_2;\b{\psi})q^{-(\la_1+\la_2)s+2\la_1}\\
=&
\sum_{\la_1\ge \la_2\ge 0} 
W_v(m_2(\varpi^{(\la_1-1)\e_1+(\la_2-1)\e_2}))W(\varpi^{\la_1\e_1+\la_2\e_2};\a_1,\a_2;\b{\psi})
q^{-(\la_1+\la_2)s+2\la_1}\\
&+
\sum_{\la_1\ge \la_2\ge 0} 
W_{v'_2}(m_2(\varpi^{\la_1\e_1+\la_2\e_2}))W(\varpi^{\la_1\e_1+\la_2\e_2};\a_1,\a_2;\b{\psi})
q^{-(\la_1+\la_2)s+2\la_1}
\end{align*}
where
\begin{align*}
W_{v'_2}(m_2(\varpi^{\la_1\e_1+\la_2\e_2}))
&=
\sum_{c\in\frak{f}}
W_v(m_2(\varpi^{\la_1\e_1+\la_2\e_2})n_2(c\varpi^{-m-1}))\\
&=
\sum_{c\in\frak{f}}
W_v(n_2(c\varpi^{\la_1+\la_2-m-1})m_2(\varpi^{\la_1\e_1+\la_2\e_2}))\\
&=
q W_v(m_2(\varpi^{\la_1\e_1+(\la_2-1)\e_2})).
\end{align*}
It follows that 
\begin{align*}
\Psi_{2,2}(\theta'(v)\ot\xi^{m+1}_{\tau,s})
=&
\sum_{\la_1\ge \la_2\ge 0} 
W_v(m_2(\varpi^{(\la_1-1)\e_1+(\la_2-1)\e_2}))W(\varpi^{\la_1\e_1+\la_2\e_2};\a_1,\a_2;\b{\psi})
q^{-(\la_1+\la_2)s+2\la_1}\\
&+
\sum_{\la_1\ge \la_2\ge 0} 
W_{v}(m_2(\varpi^{\la_1\e_1+\la_2\e_2}))W(\varpi^{\la_1\e_1+\la_2\e_2};\a_1,\a_2;\b{\psi})
q^{-(\la_1+\la_2)s+2\la_1+1}\\
=&
\sum_{\la_1\ge \la_2\ge 0} 
W_v(m_2(\varpi^{\la_1\e_1+\la_2\e_2}))W(\varpi^{(\la_1+1)\e_1+(\la_2+1)\e_2};\a_1,\a_2;\b{\psi})
q^{-(\la_1+\la_2+2)s+2(\la_1+1)}\\
&+
\sum_{\la_1\ge \la_2\ge 0} 
W_{v}(m_2(\varpi^{\la_1\e_1+\la_2\e_2}))W(\varpi^{\la_1\e_1+\la_2\e_2};\a_1,\a_2;\b{\psi})
q^{-(\la_1+\la_2)s+2\la_1+1}\\
=&
q(1+q^{-2s+1}\a_1\a_2)\Psi_{2,2}(v\ot\xi^m_{\tau,s}).
\end{align*}
This proves \eqref{E:raising theta'}.\qed

\subsubsection{Proof of \eqref{E:raising eta}}
This follows immediately from \lmref{L:eta action} and \propref{P:main prop} $(1)$.\qed\\ 

This also complete the proof of \thmref{T:main'}.\qed

\begin{remark}
Observe that \eqref{E:raising theta} and \eqref{E:raising theta'} imply
\begin{equation}\label{E:raising theta formal}
\Xi^{m+1}_{2,2}(\theta(v);X_1,X_2)
=
q(X_1+X_2)\Xi^m_{2,2}(v;X_1,X_2)
\end{equation}
and
\begin{equation}\label{E:raising theta' formal}
\Xi^{m+1}_{2,2}(\theta'(v);X_1,X_2)
=
q(1+X_1X_2)\Xi^{m}_{2,2}(v;X_1,X_2).
\end{equation}
In particular, these (when $m=a_\pi$) agree with \lmref{L:level a+1} by \eqref{E:main identity}.
\end{remark}

\subsection{Comparisons}\label{SSS:concluding remark}
We end this paper by comparing our $\cB_{\pi,m}$ in \S\ref{SS:conj basis} with that of Tasi, Casselman 
and Roberts-Schmidt. As already mentioned, when $m$ and $a_\pi$ have the opposite parity, 
our $\cB_{\pi,m}$ are slightly different from that of Tsai's. The subsets $\cB'_{\pi,m}$ appeared implicitly in 
\cite[Proposition 9.1.7]{Tsai2013} are
\[
\cB'_{\pi,m}=\stt{\eta_{\la,a_\pi+1,m}\circ\theta(v_\pi),\,\,\eta_{\lambda, a_\pi+1,m}\circ\theta'(v_\pi)\mid 
\lambda\in P^+_{G_n},\,\,2\|\lambda\|\leq m-a_\pi-1}.
\]
However, these subsets are not necessarily linearly independent as we will give a counterexample. As expected, we 
consider the case when $n=2$. In particular, the hypothesis mentioned in the beginning of this section is satisfied.\\

For a given $\la\in P^+_{\SO_4}$, let $(\sigma_\la, \cV_\la)$ be the irreducible 
representation of $\SO_4(\bbC)$ with highest weight $\la$. Denote $\chi_\la$ to be its character. 
Then $\sigma_{\epsilon_1}$ is the standard $4$-dimensional representation of $\SO_4(\bbC)$ and we have
$\bigwedge^2\sigma_{\epsilon_1}=\sigma_{\epsilon_1+\epsilon_2}\oplus\sigma_{\epsilon_1-\epsilon_2}$ 
(\cite[Theorem 5.5.13]{GoodmanWallach2009}). One checks that 
\begin{equation}\label{E:character}
\chi_{\epsilon_1}(t)=(t_1+t_2+t_1^{-1}+t_2^{-1}),
\quad
\chi_{\epsilon_1+\epsilon_2}(t)=(t_1t_2+1+t^{-1}_1t^{-1}_2),
\quad
\chi_{\epsilon_1-\epsilon_2}(t)=(t_1t^{-1}_2+1+t^{-1}_1t_2)
\end{equation}
where
\[
t
=
\begin{pmatrix}
t_1&&&\\
&t_2&&\\
&&t_2^{-1}&\\
&&&t_1^{-1}
\end{pmatrix}\in T_2(\bbC).
\]
Since $\epsilon_1$ and $\epsilon_1\pm\epsilon_2$ are minuscule weights of $\SO_4(\bbC)$, we find that 
\begin{equation}\label{E:epsilon_1}
\sS^0_{2,m}(\varphi_{\epsilon_1,m})(X_1,X_2)=q(X_1+X_2+X_1^{-1}+X_2^{-1})
\end{equation}
and
\begin{equation}\label{E:epsilon_1+2}
\sS^0_{2,m}(\varphi_{\epsilon_1+\epsilon_2},m)(X_1,X_2)=q(X_1X_2+1+X^{-1}_1X^{-1}_2)
\end{equation}
and
\begin{equation}\label{E:epsilon_1-2}
\sS^0_{2,m}(\varphi_{\epsilon_1-\epsilon_2},m)(X_1,X_2)=q(X_1X^{-1}_2+1+X^{-1}_1X_2)
\end{equation}
by \eqref{E:character} and \cite[(3.13)]{Gross1998}.\\ 

We show that $\cB'_{\pi,a_\pi+3}$ is not linearly independent when $n=2$. For this we first note that the linear maps 
$\Xi^m_{2,2}$ constructed in \propref{P:main prop} are injective by \cite[Corollary 4.3.8]{RobertsSchmidt2007},  
\eqref{E:K*}, \eqref{E:R* fix space} and \propref{P:main prop} $(3)$. Form the equations above and 
\eqref{E:raising theta formal}, \eqref{E:raising theta' formal}, \eqref{E:main identity} together with \lmref{L:eta action}, 
we find that 
\[
\Xi^{a_\pi+3}_{2,2}(\eta_{\epsilon_1,a_\pi+1,a_\pi+3}\circ\theta'(v_\pi);X_1,X_2)
\]
is equal to
\[
q\,\Xi^{a_\pi+3}_{2,2}(\eta_{0,a_\pi+1,a_\pi+3}\circ\theta(v_\pi);X_1,X_2)
+
\Xi^{a_\pi+3}_{2,2}(\eta_{\epsilon_1+\epsilon_2,a_\pi+1,a_\pi+3}\circ\theta(v_\pi);X_1,X_2).
\]
Since $\Xi^{a_\pi+3}_{2,2}$ is injective on $\cV_\pi^{K_{2,a_\pi+3}}$, this implies 
\[
\eta_{\epsilon_1,a_\pi+1,a_\pi+3}\circ\theta'(v_\pi)
=
q\,\eta_{0,a_\pi+1,a_\pi+3}\circ\theta(v_\pi)
+
\eta_{\epsilon_1+\epsilon_2,a_\pi+1,a_\pi+3}\circ\theta(v_\pi).
\]
This shows that the subsets $\cB'_{\pi,m}$ may not be linearly independent when $m$ and $a_\pi$ have the opposite 
parity.\\ 

Next we compare $\cB_{\pi,m}$ (when $n=2$) with those given by \eqref{E:oldform}. We show that they are different 
in general. In fact, this already happens when $m=a_\pi+2$. Suppose that $\Lambda_{\pi,\psi}(v_\pi)=1$. 
Then under $\Xi^{a_\pi+2}_{2,2}$, the basis for $\cV_\pi^{K_{2,a_\pi+2}}$ given by \eqref{E:oldform} turns to
\[
\stt{qX_1X_2,\,\,q^2(X_1+X_2)(1+X_1X_2),\,\,q^2(1+2X_1X_2+X_1^2X_2^2),\,\,q^2(X^2_1+2X_1X_2+X_2^2)}
\]
by \eqref{E:raising theta formal}, \eqref{E:raising theta' formal}, \eqref{E:main identity} and \lmref{L:eta action}.
On the other hand, under $\Xi^{a_\pi+2}_{2,2}$, the subset $\cB_{\pi,a_{\pi+2}}$ becomes
\[
\stt{qX_1X_2,\,\,q^2(X_1+X_2)(1+X_1X_2),\,\,q^2(1+X_1X_2+X_1^2X_2^2),\,\,q^2(X_1^2+X_1X_2+X_2^2)}
\]
by \eqref{E:epsilon_1}, \eqref{E:epsilon_1+2}, \eqref{E:epsilon_1-2} and \lmref{L:eta action}. Since $\Xi^{a_\pi+2}_{2,2}$
is injective on $\cV_\pi^{K_{2,a_\pi+2}}$, our claim follows.\\

Finally, we remark that our $\cB_{\pi,m}$ are different from that of Casselman implicitly given in the proof
of \cite[Theorem 1]{Casselman1973} when $m$ and $a_\pi$ have the opposite parity.

\subsection*{Acknowledgements}
Part of this works was done during the author's visit to the Mathematics Department at National University of Singapore. 
He thanks the Mathematics Department for offering office space and access to computing and internet services.
The author would like to thank Wee Teck Gan for his hospitality during the author's visit and for drawing 
the author's attention to Tasi's thesis as well as various helpful suggestions. Thanks also due to Huanchen Bao, Rui Chen,
Chuijia Wang, Xiaolei Wan, Jialiang Zou and Lei Zhang for their hospitality and helps during the author's stay.   
The author also like to thank Ming-Lun Hsieh for his encouragement and Shih-Yu Chen for the discussions. 
Finally, the author would like the thank the referee, who carefully read the original manuscript and provided lots of 
helpful comments and suggestions. This work is partially supported by the Postdoctoral Research Abroad Program of 
Ministry of Science and Technology with the grant number 108-2917-I-564-008.

\bibliographystyle{alpha}
\bibliography{ref}

\begin{thebibliography}{GPSR87}

\bibitem[AL70]{AtkinLehner1970}
A.O.L. Atkin and J.~Lehner.
\newblock {Hecke operators on $\Gamma_0(m)$}.
\newblock {\em Mathematische Annalen}, 185:134--160, 1970.

\bibitem[Art81]{Arthur1981}
J.~Arthur.
\newblock {Automorphic Representations and Number Theory}.
\newblock In {\em {Seminar on Harmonic Analysis}}, volume~1, pages 3--51.
  Canadian Mathematical Society, 1981.

\bibitem[BK14]{BrumerKenneth2014}
A.~Brumer and K.~Kenneth.
\newblock {Paramodular abelian varieties of odd conductor}.
\newblock {\em Transactions of the American Mathematical Society},
  366(5):2463--2516, 2014.

\bibitem[BZ76]{BZ1976}
I.~N. Bern{\v s}te{\u \i}n and A.~V. Zelevinski{\u \i}.
\newblock {Representations of the group ${\rm GL}(n,F)$, where $F$ is a local
  non-Archimedean field}.
\newblock {\em Uspekhi Matematicheskikh Nauk}, 31(3):5--70, 1976.

\bibitem[Cas73]{Casselman1973}
W.~Casselman.
\newblock {On some results of Atkin and Lehner}.
\newblock {\em Mathematische Annalen}, 201:301--314, 1973.

\bibitem[CS80]{CasselmanShalika1980}
W.~Casselman and J.~Shalika.
\newblock {The unramified principal series of $p$-adic groups. II. The
  Whittaker function}.
\newblock {\em Compositio Mathematica}, 41(2):207--231, 1980.

\bibitem[CST17]{CogdellShahidiTsai2017}
J.~W. Cogdell, F.~Shahidi, and T.-L. Tsai.
\newblock {Local Langlands correspondence for ${\rm GL}_n$ and the exterior and
  symmetric square $\epsilon$-factors}.
\newblock {\em Duke Mathematical Journal}, 166(11):2053--2132, 2017.

\bibitem[Gin90]{Ginzburg1990}
D.~Ginzburg.
\newblock {$L$-functions for ${\rm SO}_n\times{\rm GL}_k$}.
\newblock {\em Journal fur die reine und angewandte Mathematik}, 405:156--180,
  1990.

\bibitem[GPSR87]{GPSR1987}
S.~Gelbart, I.~Piatetski-Shapiro, and S.~Rallis.
\newblock {\em {Explicit constructions of automorphic $L$-functions}}.
\newblock Number 1254 in Lecture Notes in Mathematics. Springer-Verlag, Berlin,
  1987.

\bibitem[Gro98]{Gross1998}
B.~Gross.
\newblock {On the Satake isomorphism}.
\newblock In {\em {Galois Representations in Arithmetic Algebraic Geometry}},
  {London Mathematical Society Lecture Note Series}, pages 223--238. {Cambridge
  University Press, Cambridge}, 1998.

\bibitem[Gro15]{Gross2015}
B.~Gross.
\newblock {On the Langlands correspondence for symplectic motives}, 2015.

\bibitem[GW09]{GoodmanWallach2009}
R.~Goodman and N.~Wallach.
\newblock {\em {Symmetry, Representations, and Invariants}}, volume 255 of {\em
  {Graduate Texts in Mathematics}}.
\newblock Springer, 2009.

\bibitem[Hen00]{Henniart2000}
G.~Henniart.
\newblock {Une preuve simple des conjectures de Langlands pour ${\rm GL}(n)$
  sur un corps $p$-adique}.
\newblock {\em Inventiones Mathematiace}, 139:439--455, 2000.

\bibitem[HT01]{HarrisTaylor2001}
M.~Harris and R.~Taylor.
\newblock {\em {On the geometry and cohomology of some simple Shimura
  varieties}}, volume 151 of {\em Annals of Mathematics Studies}.
\newblock Princeton University press, 2001.

\bibitem[Jac12]{Jacquet2012}
H.~Jacquet.
\newblock {A correction to Conducteur des repr{\'e}sentations du groupe
  lin{\'e}aire}.
\newblock {\em Pacific Journal of Mathematics}, 260(2):515--525, 2012.

\bibitem[JL70]{JLbook}
H.~Jacquet and R.~Langlands.
\newblock {\em Automorphic Forms on ${\rm GL}(2)$}, volume 114 of {\em Lecture
  Notes in Mathematics}.
\newblock Springer-Verlag, Berlin and New York, 1970.

\bibitem[JPSS81]{JPSS1981}
H.~Jacquet, I.~Piatetski-Shapiro, and J.~Shalika.
\newblock {Conducteur des repr{\'e}sentations du groupe lin{\'e}aire}.
\newblock {\em Mathematische Annalen}, 256(2):199--214, 1981.

\bibitem[JPSS83]{JPSS1983}
H.~Jacquet, I.~Piatetski-Shapiro, and J.~Shalika.
\newblock {Rankin-Selberg convolutions}.
\newblock {\em American Journal of Mathematics}, 105(2):367--464, 1983.

\bibitem[JS83]{JacquetShalika1983}
H.~Jacquet and J.~Shalika.
\newblock {The Whittaker models of induced representations}.
\newblock {\em Pacific Journal of Mathematics}, 109(1):107--120, 1983.

\bibitem[JS03]{JiangSoudry2003}
D.~Jiang and D.~Soudry.
\newblock {The local converse theorem for ${\rm SO}(2n+1)$ and applications}.
\newblock {\em Annals of Mathematics}, 157(3):743--806, 2003.

\bibitem[JS04]{JiangSoudry2004}
D.~Jiang and D.~Soudry.
\newblock {Generic representations and local Langlands recprocity law for
  $p$-adic ${\rm SO}_{2n+1}$}.
\newblock In {\em Contributions to automorphic forms, goemetry, and number
  theory}, pages 457--519. Johns Hopkins Univ. Press, Baltimore, MD, 2004.

\bibitem[Kap13]{Kaplan2013}
E.~Kaplan.
\newblock {On the gcd of local Rankin-Selberg integrals for even orthogonal
  groups}.
\newblock {\em Compositio Mathematica}, 149(4):587--636, 2013.

\bibitem[Kap15]{Kaplan2015}
E.~Kaplan.
\newblock {Complementary results on the Rankin-Selberg gamma factors of
  classical groups}.
\newblock {\em Journal of Number Theory}, 146:390--447, 2015.

\bibitem[Mat13]{Matringe2013}
N.~Matringe.
\newblock {Essential Whittaker functions for $GL(n)$}.
\newblock {\em Documenta Mathematica}, 18:1191--1214, 2013.

\bibitem[Miy14]{Miyauchi2014}
M.~Miyauchi.
\newblock {Whittaker functions associated to newforms for $GL(n)$ over $p$-adic
  fields}.
\newblock {\em Journal of the Mathematical Society of Japan}, 66(1):17--24,
  2014.

\bibitem[Nov79]{Novodvorsky1979}
M.~Novodvorsky.
\newblock {Automorphic $L$-functions for the symplectic group ${\rm GSp}(4)$}.
\newblock In {\em {Automorphic Forms, Representations and L-Functions, Part
  2}}, volume~33, pages 87--95. Preceedings of Symposium in Pure Mathematics,
  Oregon State University, 1979.

\bibitem[Ree91]{Reeder1991}
M.~Reeder.
\newblock {Old Forms on ${\rm GL}_n$}.
\newblock {\em American Journal of Mathematics}, 113(5):911--930, 1991.

\bibitem[RS07]{RobertsSchmidt2007}
B.~Roberts and R.~Schmidt.
\newblock {\em {Local newforms for ${\rm GSp}(4)$}}.
\newblock Number 1918 in Lecture Notes in Mathematics. Springer-Verlag, Berlin
  Heidelberg, 2007.

\bibitem[Sat63]{Satake1963}
I.~Satake.
\newblock {Theory of spherical functions on reductive algebraic groups over
  $p$-adic fields}.
\newblock {\em Publications Math{\'e}matiques de l'IH{\'E}S}, 18:5--69, 1963.

\bibitem[Sch02]{Schmidt2002}
R.~Schmidt.
\newblock {Some remarks on local newforms for ${\rm GL}(2)$}.
\newblock {\em Journal of Ramanujan Mathematical Society}, 17(2):115--147,
  2002.

\bibitem[Sch13]{Scholze2013}
P.~Scholze.
\newblock {The local Langlands correspondence for ${\rm GL}_n$ over $p$-adic
  local fields}.
\newblock {\em Inventiones Mathematiace}, 192(3):663--715, 2013.

\bibitem[Sha90]{Shahidi1990}
F.~Shahidi.
\newblock {A proof of Langlands' conjecture on Plancherel measures:
  Complementary series of $p$-adic groups}.
\newblock {\em Annals of Mathematics}, 132(2):273--330, 1990.

\bibitem[Sha18]{Shahabi2018}
M.~Shahabi.
\newblock {\em Smooth integral models for certain congruence subgroups of odd
  spin groups}.
\newblock PhD thesis, University of Calgary, 2018.

\bibitem[Shi76]{Shintani1976}
T.~Shintani.
\newblock {On an explicit formula for class-1 "Whittaker functions'' on ${\rm
  GL}_n$ over $\mathfrak{P}$-adic fields}.
\newblock {\em Proceedings of the Japan Academy}, 52(4):180--182, 1976.

\bibitem[Sou93]{Soudry1993}
D.~Soudry.
\newblock {\em {Rankin-Selberg convolution for ${\rm SO}_{2\ell+1}\times{\rm
  GL}_n$ : Local theory}}, volume 105.
\newblock Memoirs of the American Mathematical Society, 1993.

\bibitem[Sou00]{Soudry2000}
D.~Soudry.
\newblock {Full multiplicativity of gamma factors for ${\rm
  SO}_{2\ell+1}\times{\rm GL}_n$}.
\newblock In {\em {Proceedings of the Conference on $p$-adic Aspects of the
  Theory of Automorphic Representations (Jerusalem, 1998)}}, volume 120, pages
  511--561. Israel Journal of Mathemtaics, Part B, 2000.

\bibitem[Tat79]{Tate1979}
J.~Tate.
\newblock Number theoretic background.
\newblock In {\em {Automorphic forms, representations, and $L$-functions, Part
  2}}, volume~33, pages 3--26. Preceedings of Symposium in Pure Mathematics,
  Oregon State University, 1979.

\bibitem[TB00]{Takloo-Bighash2000}
R.~Takloo-Bighash.
\newblock {$L$-functions for $p$-adic group $GSp(4)$}.
\newblock {\em American Journal of Mathematics}, 122(6):1085--1120, 2000.

\bibitem[Tit79]{Tits1979}
J.~Tits.
\newblock {Reductive groups over local fields}.
\newblock In {\em {Automorphic forms, representations, and $L$-functions, Part
  1}}, volume~33, pages 29--69. Preceedings of Symposium in Pure Mathematics,
  Oregon State University, 1979.

\bibitem[Tsa13]{Tsai2013}
P-Y Tsai.
\newblock {\em On newforms for split special odd orthogonal groups}.
\newblock PhD thesis, Harvard University, Cambridge, Massachusetts, 2013.

\bibitem[Tsa16]{Tsai2016}
P-Y Tsai.
\newblock Newforms for odd orthogonal groups.
\newblock {\em Journal of Number Theory}, 161:75--87, 2016.

\end{thebibliography}
\end{document}